\documentclass[12pt]{article}
\usepackage[margin=1in]{geometry} 
\usepackage{amsmath,amsthm,amssymb,amsfonts}
\usepackage{booktabs}
\numberwithin{equation}{section}

\usepackage{enumitem}
\usepackage[utf8]{inputenc}
\usepackage[pagewise]{lineno}

\usepackage{hyperref}
\usepackage{csquotes}
\usepackage[english]{babel}
\usepackage{xcolor}

\providecommand{\keywords}[1]
{
  \small	
  \textbf{\textit{Keywords:}} #1
}

\newcommand{\MSC}[1]{%
  \small
  \textbf{\textit{Mathematics Subject Classification:}} #1
}

\title{Boundedness in a chemotaxis system with nonlinear diffusion and singular sensitivity}
\author{
    Minh Le\thanks{Institute for Theoretical Sciences, Westlake University, China \texttt{(leminh@westlake.edu.cn)}} 
}
\date{}

\begin{document}
\maketitle

\begin{abstract}
This paper studies the following chemotaxis system with nonlinear diffusion and singular sensitivity
\[
\begin{cases}
    u_t = \nabla \cdot \left( (u+a)^{m-1} \nabla u \right)
    - \chi \nabla \cdot \left( u \dfrac{\nabla v}{v^k} \right), \\[4pt]
    v_t = \Delta v - v + u,
\end{cases}
\]
in a smoothly bounded domain \(\Omega \subset \mathbb{R}^n\), \(n \geq 2\), subject to no-flux boundary conditions, where \(\chi,k,a,m>0\). For suitably regular initial data, we prove global existence and uniform-in-time boundedness of classical solutions if one of the following conditions holds:
\[
\begin{aligned}
& n=2,\quad k\in(0,\infty),\quad m>1;\\
& n\ge 3,\quad k\in\left(0,\tfrac12\right),\quad
  m>1+\frac{(1-k)(n-2)}{n};\\
& n\ge 3,\quad k\in\left[\tfrac12,1\right),\quad
  m>\frac{3n-2}{2n};\\
& n\ge 3,\quad k\in[1,\infty),\quad
  m>\max\left\{1,\frac{3n-2}{2n+2}\right\}.
\end{aligned}
\]
The main contribution is the physically relevant three-dimensional case. In particular, for every \(\chi>0\), \(a>0\), \(m>1\) and \(k\ge 1\), the system admits a global bounded classical solution. Thus, in dimension three, nonlinear diffusion of porous-medium type removes the smallness restriction on the chemotactic sensitivity $\chi$ that earlier results require in the case \(m=1\) and \(k\geq 1\).
\end{abstract}

\\
\MSC{35B35, 35K45, 35K55, 92C15, 92C17}

\numberwithin{equation}{section}
\newtheorem{theorem}{Theorem}[section]
\newtheorem{lemma}[theorem]{Lemma}
\newtheorem{remark}{Remark}[section]
\newtheorem{Prop}{Proposition}[section]
\newtheorem{Def}{Definition}[section]
\newtheorem{Corollary}{Corollary}[theorem]
\allowdisplaybreaks

\section{Introduction}

We are concerned with the following chemotaxis system with nonlinear diffusion and singular sensitivity in a smoothly bounded domain \(\Omega \subset \mathbb{R}^n\) with \(n \geq 2\):
\begin{equation} \label{1}
\begin{cases}
    u_t = \nabla \cdot \left(  (u+a)^{m-1}  \nabla u \right) - \chi \nabla \cdot \left( u \dfrac{\nabla v}{v^k} \right), \qquad &\text{in }\Omega \times (0, \infty), \\[6pt]
    v_t = \Delta v - v + u, \qquad &\text{in }\Omega \times (0, \infty), 
\end{cases}
\end{equation}
where  \(\chi>0\), \(k>0\), \(a>0\) and \(m >1\). The system is endowed with the homogeneous Neumann boundary condition
\begin{equation} \label{bdry}
    \frac{\partial u}{\partial \nu} =\frac{\partial v}{\partial \nu} =0 ,\qquad x \in \partial \Omega, \, t>0,
\end{equation}
and initial conditions
\begin{equation} \label{initial}
    \begin{cases}
    \displaystyle u_0 \in C^{0} (\bar{\Omega}) \text{ is nonnegative with \(\int_\Omega u_0 >0\),}  \\
    \displaystyle v_0 \in W^{1, \infty}(\Omega) \;\; \text{and \(v_0> 0\) in \(\bar{\Omega}\).} 
    \end{cases}
\end{equation}

System \eqref{1} is used to model the directed movement of cells or micro-organisms in response to chemical stimuli, a phenomenon known as chemotaxis. It belongs to the well-known class of Keller--Segel systems, which were originally introduced in \cite{Keller2} for the regular sensitivity case (\(m=1\) and \(k=0\)) and in \cite{Keller} for the singular sensitivity case (\(m=1\) and \(k=1\)). Here, the functions \(u(x,t)\) and \(v(x,t)\) represent the population density and the chemical signal concentration at position \(x\) and time \(t\), respectively.

The regular sensitivity Keller--Segel system has attracted considerable attention from the mathematical community, largely due to one of its most striking properties, known as the \emph{critical mass} phenomenon. In two dimensions, this property states that if the initial mass is less than a critical value, solutions remain global \cite{Dolbeault, Dolbeault1, NSY}, whereas if the initial mass exceeds this value, solutions may blow up in finite time \cite{Nagai1, Nagai2, Nagai3, Nagai4}. However, this phenomenon does not persist in higher-dimensional settings. Indeed, it was shown in \cite{Winkler-2010, Winkler-2013} that in dimensions \(n \geq 3\), there exist initial data with arbitrarily small mass that still lead to finite-time blow-up.

In contrast, the singular sensitivity Keller--Segel system does not exhibit this critical mass phenomenon. Indeed, it was shown in \cite{Nagai+Senba} that solutions of the parabolic--elliptic system exist globally when \(\chi< \frac{2}{(n-2)_+}\), whereas finite-time blow-up occurs when \(\chi> \frac{2n}{(n-2)_+}\) under the assumption that the initial data are radially symmetric. Without this symmetry assumption, a global existence result for the parabolic--elliptic system was established in \cite{Biler-1999}, while global boundedness was proven in \cite{KMT} under the assumption that \(\chi< \frac{2}{n}\). Parallel to the parabolic--elliptic case, the global existence result for the fully parabolic case was established in \cite{Winkler+2011}, while global boundedness was proved in \cite{Fujie} under the assumption that \(\chi< \sqrt{\frac{2}{n}}\). In \cite{Fujie_2018}, it was shown that the condition \(\chi< \frac{n}{n-2}\) ensures global boundedness when the left-hand side of the first equation in \eqref{1} is replaced by \(\kappa u_t\) with a sufficiently small \(\kappa>0\). Moreover, it was conjectured in that paper that if \(\chi> \frac{n}{n-2}\), solutions may blow up in finite time. Without the smallness assumption on \(\kappa\), it was shown in \cite{Lankeit_2016, Le2026NA} that there exists \(\chi_0> \sqrt{\frac{2}{n}}\) such that for \(\chi < \chi_0\), solutions are global and bounded.

We now shift our attention to the nonlinear diffusion case where \(m>1\). In the regular sensitivity case where \(k=0\), the global existence and boundedness of solutions were demonstrated in \cite{TaoWinkler2012} when \(m> 2- \frac{2}{n}\), while finite-time blow-up of radial solutions may occur when \(m < 2- \frac{2}{n}\) \cite{CS2012}. This raises a natural question: in the singular sensitivity case where \(k>0\), does there exist a critical value \(m_c>0\) such that for \(m>m_c\), solutions are global and bounded, while for \(m< m_c\), solutions may blow up in finite time?

In the parabolic--elliptic case, where the second equation is replaced by \(\Delta v-v+u=0\), this question has been partially answered. More specifically, it was shown in \cite{WangDingLi2017} that one can take \(m_c \leq 1+ \frac{(n-2)(1-k)}{n}\) when \(k < 1\), and \(m_c \leq 1\) when \(k \geq 1\). However, the understanding of the fully parabolic system \eqref{1} is very limited. To the best of our knowledge, the only existing result is \cite{Ren2024}, where the global existence and boundedness of weak solutions were established in dimension two for \(m> \frac{3}{2}\), which is far from satisfactory when compared to the parabolic--elliptic case. Therefore, the main objective of this paper is to provide a partial answer to this question regarding the global boundedness side for the fully parabolic system \eqref{1}. We now precisely state our main result in the following theorem.

\begin{theorem} \label{thm1}
Let \(\Omega \subset \mathbb{R}^n\), \(n \geq 2\), be a smoothly bounded domain, assumed in addition to be convex when \(n \ge 3\). Assume that \(a>0\), \(\chi>0\) and suppose that one of the following conditions holds:
\begin{align}
    &n = 2, \quad k \in (0,\infty), \quad m > 1; \tag{A} \label{case:A} \\[2pt]
    &n \geq 3, \quad k \in \left(0,\frac{1}{2}\right), \quad
        m > 1 + \frac{(1-k)(n-2)}{n}; \tag{B} \label{case:B} \\[2pt]
    &n \geq 3, \quad k \in \left[\frac{1}{2},1\right), \quad
        m > \frac{3n-2}{2n}; \tag{C} \label{case:C} \\[2pt]
    &n \geq 3, \quad k \in [1,\infty), \quad
        m > \max\left\{1,\frac{3n-2}{2n+2}\right\}. \tag{D} \label{case:D}
\end{align}
    Then, under the boundary condition \eqref{bdry}, the system \eqref{1} with initial data \eqref{initial} possesses a unique global classical positive solution \((u,v)\) such that
    \[
    u, v \in C^0 \left(\bar{\Omega} \times [0,\infty)\right) \cap C^{2,1} \left(\bar{\Omega} \times (0,\infty)\right).
    \]
    Moreover, the solution is uniformly bounded in the sense that
    \begin{align*}
        \sup_{t > 0} \left\{ \| u(\cdot,t) \|_{L^\infty(\Omega)} + \| v(\cdot,t) \|_{W^{1,\infty}(\Omega)} \right\} < \infty.
    \end{align*}
\end{theorem}

\begin{remark}
    Table~\ref{tab:comparison} compares the parameter ranges for which global
existence and boundedness are known for the parabolic--elliptic system of
\cite{WangDingLi2017}, in which the second equation of \eqref{1} is replaced
by \(\Delta v-v+u=0\), with those established for the fully parabolic system
\eqref{1} in Theorem~\ref{thm1}. 
\begin{table}[htbp]
\centering
\caption{Comparison between the parabolic--elliptic system of
\cite{WangDingLi2017} and the fully parabolic system \eqref{1}.}
\label{tab:comparison}
\renewcommand{\arraystretch}{1.4}
\small
\begin{tabular}{|p{2.5cm}|p{2.2cm}|p{4.4cm}|p{4.6cm}|}
\hline
System & Status & Condition (i): \(m>1,\ k\ge 1\)
& Condition (ii): \(m>1+\dfrac{(1-k)(n-2)}{n},\ k\in(0,1)\) \\
\hline
Parabolic--elliptic
\cite{WangDingLi2017}
& solved
&  \(n\ge 2\)
& \(n\ge 2\) \\
\hline
Fully parabolic
(Theorem~\ref{thm1})
& solved
& \(n=2\) \ (case \eqref{case:A}); $n=3,4$\ (case \eqref{case:D})
& \(n=2\) \ (case \eqref{case:A});
  \(n\ge 3,\ k\in(0,\tfrac12]\) \ (case \eqref{case:B});
  \(n\ge 3,\ k\in[\tfrac12,1),\ m>\tfrac{3n-2}{2n}\) \ (case \eqref{case:C}) \\
\hline
Fully parabolic
& open
& \(n\ge 5\)
& \(n\ge 3,\ k\in(\tfrac12,1),\ 
  1+\dfrac{(1-k)(n-2)}{n}<m\le\dfrac{3n-2}{2n}\) \\
\hline
\end{tabular}
\end{table}
\end{remark}
\begin{remark}
We establish global existence and boundedness for all \(\chi>0\), \(m>1\), \(k\ge 1\), and \(n=3,4\). By contrast, in the case \(m=1\) and \(k=1\), earlier work in \cite{Winkler+2011} and \cite{Le2026NA} showed that solutions are global and bounded only when \(\chi\) is sufficiently small. In the physically relevant dimension \(n=3\), the assumptions \(m>1\) and \(a>0\) allow us to remove this smallness restriction on \(\chi\). 
\end{remark}

\begin{remark}
    The argument used in the proof of Theorem \ref{thm1} also works for the case \(k=0\), and this provides a different proof compared to \cite{TaoWinkler2012}, which proves the global boundedness of solutions for \(m>\frac{2n-2}{n}\).
\end{remark}

\begin{remark}
    Our results indicate that the singular sensitivity is indeed in favor of blow-up prevention since compared to the regular sensitivity case where \(k=0\), the range of \(m\) for global boundedness results is wider.
\end{remark}

\textbf{Challenges and main ideas.}
The main difficulty in Theorem~\ref{thm1} is to handle the simultaneous presence of singular sensitivity and nonlinear diffusion. On the one hand, unlike the parabolic–elliptic case, the approach used in \cite{WangDingLi2017} is not applicable to the fully parabolic case, since we cannot replace the Laplacian of the chemical signal by the cell density. On the other hand, the classical technique for the linear diffusion case \(m=1\), which is based on the evolution of
\(\int_\Omega u^p v^{-q}\) for suitable \(p,q>0\) \cite{Winkler+2011}, does not carry over to \(m>1\).
Moreover, the methods developed in \cite{Ding2018,Jia2019} for nonlinear diffusion with singular sensitivity are not directly applicable to the present setting.
The most delicate and technically demanding case is \eqref{case:D},
whose main idea we describe below.
Intuitively, singular sensitivity and signal consumption exhibit common
structural features: for example, the approach in \cite{Winkler+2011}
for singular sensitivity and that in \cite{Tao2011} for signal consumption
are both based on the evolution of weighted quantities of the form
\(\int_\Omega u^p\phi(v)\).
Following this intuition, and inspired by the approach in \cite{Winkler2022}
for a chemotaxis system with nonlinear diffusion and signal consumption in
dimension three, we introduce the energy functional
\[
F(t)=\int_\Omega u^{1+m}
+A\int_\Omega u\frac{|\nabla v|^2}{v}
+B\int_\Omega\frac{|\nabla v|^4}{v^2},
\qquad t\in(0,T_{\max}),
\]
where the weights \(A>0\) and \(B>0\) are selected so that the evolutions of the three
components of \(F\) mutually absorb each other's unfavorable terms.
To make this precise, set
\[
\mathcal E_0:=\int_\Omega (u+a)^{m+1},\qquad
\mathcal E_1:=\int_\Omega u\frac{|\nabla v|^2}{v},\qquad
\mathcal E_2:=\int_\Omega \frac{|\nabla v|^4}{v^2}.
\]
In the evolution of \(\mathcal E_0\), the nonlinear diffusion contributes the
favorable term \(-m\mathcal D\), where
\[
\mathcal D:=\int_\Omega (u+a)^{2m-2}|\nabla u|^2,
\]
while the singular sensitivity contributes the unfavorable term
\[
\mathcal G:=\int_\Omega u^2\frac{|\nabla v|^2}{v^2}.
\]
By Lemma~\ref{LM}, the evolution of \(\mathcal E_1\) produces the favorable
term \(-\frac12\mathcal G\), which absorbs the unfavorable term \(\mathcal G\)
arising from \(\mathcal E_0\). At the same time, it produces a small multiple
\(\varepsilon\mathcal D\) of \(\mathcal D\), which is absorbed by the diffusion
term in \(\mathcal E_0\), as well as the unfavorable term
\[
\mathcal B:=\int_\Omega \frac{|\nabla v|^2|D^2v|^2}{v^2}
+\int_\Omega \frac{|\nabla v|^6}{v^4}.
\]
By Lemma~\ref{L2}, the term \(\mathcal B\) is dominated by the dissipative
contribution from the evolution of \(\mathcal E_2\). Finally, again by
Lemma~\ref{L2}, the evolution of \(\mathcal E_2\) produces the unfavorable term
\[
\mathcal R:=\int_\Omega |\nabla u|^2,
\]
which is absorbed, via Young's inequality, into the nonlinear diffusion term
\(-m\mathcal D\) from \(\mathcal E_0\), at the cost of the relaxation term
\[
\int_\Omega \frac{|\nabla u|^2}{(u+a)^2}.
\]
The second most delicate case is the one in which \eqref{case:B} holds.
To treat this case, we consider the energy functional
\[
G(t):=\frac{1}{p}\int_\Omega u^p
+\int_\Omega \frac{|\nabla v|^{2\left(p-m+1+\frac{1-2k}{\lambda}\right)}}
{v^{p-m+1+\frac{1-2k}{\lambda}}}
+\int_\Omega v^{\lambda(p-m+1)+2-2k},
\]
where \(1\le \lambda<\frac{n}{n-2}\) is suitably chosen, and apply the key
inequality established in Lemma~\ref{L} to establish the \(L^p\) boundedness
of \(u\). The case \eqref{case:C} is treated by an analogous argument, whereas
\eqref{case:A} follows from a standard energy method. \\
The remainder of the paper is organized as follows. In Section~\ref{S2}, we
recall several useful inequalities and a standard parabolic regularity result
that will be applied in subsequent sections. Section~\ref{S3} establishes the
local well-posedness of solutions to system \eqref{1} and derives several a
priori estimates. The \(L^p\) boundedness of solutions is proved in
Section~\ref{S4} for dimension two and in Section~\ref{S5} for higher
dimensions. Finally, Section~\ref{S6} provides the proof of
Theorem~\ref{thm1}.

\section{Preliminaries}\label{S2}

In this section, we recall several useful inequalities and a standard parabolic regularity theory which will be applied in the subsequent sections. Let us commence this section by presenting a generalized form of the Gagliardo--Nirenberg interpolation inequality, whose proof can be found in \cite[Lemma 2.3]{LiLankeit2016}.
\begin{lemma} \label{GN}
Let \(\Omega\) be a  bounded and smooth domain in \(\mathbb{R}^n\) with \(n \geq 1\). Let \(r \geq 1\), \(0< q\leq p < \infty\), \(s>0\) such that 
\begin{align*}
    \frac{1}{r} \leq \frac{1}{n}+\frac{1}{p}.
\end{align*} Then there exists a constant \(C_{GN}>0\) such that 
\begin{equation*}
    \left \| f \right \|^p_{L^p(\Omega)}\leq C_{GN}\left ( \left \| \nabla f \right \|_{L^r(\Omega)}^{p \theta}\left \| f \right \|^{p(1-\theta)}_{L^q(\Omega)} +\left \| f \right \|^p_{L^s(\Omega)} \right )\qquad \text{for all } f \in W^{1,r}(\Omega)\cap L^q(\Omega),
\end{equation*}
where \(\theta= \frac{\frac{1}{q}-\frac{1}{p}}{\frac{1}{q}+\frac{1}{n}-\frac{1}{r}} \in [0,1]\).
\end{lemma}

To prepare for the proof of the \(L^p\) boundedness in the case \(m>1\) and \(k\ge 1\), we will need the following inequalities, whose proofs are given in \cite[Lemma 2.5]{Le2026NA}.
 \begin{lemma} \label{Lm}
     Let \(\Omega \subset \mathbb{R}^n\), \(n\geq 1\), be a smoothly bounded domain and let \(l>0\) and \(q > l+1\). For any \(\phi \in C^2(\bar{ \Omega  })\) such that \(\phi>0\) in \(\bar{\Omega}\) and \(\frac{\partial \phi}{\partial \nu}=0\) on \(\partial \Omega\), the following inequalities hold
        \begin{align} \label{Lm.1}
             \int_\Omega \phi^{-l-2}|\nabla \phi|^{q+2} \leq \left ( \frac{q+\sqrt{n}}{l+1} \right )^2 \int_\Omega \phi^{-l}|\nabla \phi|^{q-2} |D^2 \phi|^2,
        \end{align}
        and
    \begin{align} \label{Lm.2}
        \int_\Omega \phi^{-l-2}|\nabla \phi|^{q+2} \leq \left ( 4\left ( \frac{l+1+\sqrt{n}}{q-l+1} \right )^2+2 \right )\int_\Omega \phi^{-l+2}|\nabla \phi|^{q-2}|D^2 \ln \phi|^2,
    \end{align}
    as well as
    \begin{align} \label{Lm.3}
        \int_\Omega \phi^{-l}|\nabla \phi|^{q-2} |D^2 \phi|^2 \leq \left ( \frac{q-l+1}{q-l-1} \cdot \left ( \frac{l+1+\sqrt{n}}{q-l+1} \right )^2+1 \right )\int_\Omega \phi^{-l+2}|\nabla \phi|^{q-2}|D^2 \ln \phi|^2.
    \end{align}
    \end{lemma}

The next lemma provides an elementary yet very useful differential inequality, which allows us to exploit the spatio-temporal information of the solutions. For a detailed proof, we refer the reader to \cite[Lemma 3.4]{Winkler+2019}.
\begin{lemma}\label{ODI}
    Let \(\mu>0\), \(t_0\in\mathbb R\), \(T\in(t_0,\infty]\), and suppose that the nonnegative function
    \(h\in L^1_{\rm loc}([t_0,T))\) has the property that there exist \(\tau\in(0,T-t_0)\) and \(\gamma>0\) such that
    \[
    \frac1{\tau}\int_t^{t+\tau}h(s)\,ds \leq \gamma     \qquad\text{for all }t\in[t_0,T-\tau).
    \]
    Assume that \(y\in C^0([t_0,T))\cap C^1((t_0,T))\) satisfies
    \[
        y'(t)+\mu y(t)\leq h(t) \qquad\text{for all }t\in(t_0,T).
    \]
    Then
    \[
        y(t)\leq y(t_0)e^{-\mu (t-t_0)} +\frac{\gamma \tau}{1-e^{-\mu \tau}}  \qquad\text{for all }t\in[t_0,T).
    \]
\end{lemma}

The following lemma is the key tool for handling the singular behavior of \(v\) near zero. For \(p\ge 3\), the result follows immediately from Proposition~1.3 in \cite{Kurt+Shen} or from Proposition~3.1 in \cite{Kurt}. The broader case \(p>1\) is covered by a different proof given in \cite[Lemma~2.4]{minh-preprint-2}.

 \begin{lemma} \label{LK-1}
    Let \(\Omega \subset \mathbb{R}^n\) with \(n \geq 2\) be a bounded domain with smooth boundary. For any \(p >1\), there exist positive constants \(M_1=M_1(p)\) and \(M_2=M_2(p)\) such that the following inequality holds
    \begin{align*} 
        \int_\Omega \frac{|\nabla w|^{2p}}{w^p} \leq M_1 \int_\Omega |\Delta w|^p +M_2\int_\Omega w^p,
    \end{align*}
    for any positive function \(w \in C^2(\bar{\Omega})\) satisfying \(\frac{\partial w}{\partial \nu}=0\) on \(\partial \Omega\).
\end{lemma}

  Next, we recall a standard \(L^p\) regularity result for parabolic equations as established in \cite[Lemma 4.1]{DM}.
\begin{lemma} \label{C52.Para-Reg}
Let \(\Omega \subset \mathbb{R}^n\) with \(n \geq 2\) be a bounded domain with smooth boundary. Assume that \(p\geq 1\) and \(q \geq 1\) satisfy 
\begin{equation*}
    \begin{cases}
     q &< \frac{np}{n-p},  \qquad \text{when } p<n,\\
     q &< \infty, \qquad \text{when } p=n,\\
      q &= \infty, \qquad \text{when } p>n.\\
     \end{cases}
\end{equation*}
{Assume that $V_0 \in W^{1,q}(\Omega)$, $\mu>0$ and $T\in (0, \infty]$. Let $f\in C^0 \left ( \bar{\Omega}\times[0, T)  \right )$, and let $V\in C^0 \left ( \bar{\Omega}\times[0,  T)  \right ) \cap C^{2,1}\left ( {\Omega}\times(0,  T)  \right ) $ be a classical solution of}
\begin{equation}\label{C52.parabolic-equation}
    \begin{cases}
     V_t = \Delta V  - \mu V + f &\text{in } \Omega \times (0,T), \\ 
\frac{\partial V}{\partial \nu} =  0 & \text{on }\partial \Omega \times (0,T),\\ 
 V(\cdot,0)=V_0   & \text{in } \Omega.
    \end{cases}
\end{equation}
 If \(f \in L^\infty \left ( (0,T);L^p(\Omega) \right ) \), then \(V  \in L^\infty \left ( (0,T);W^{1,q}(\Omega) \right )\).
\end{lemma}

\section{Local existence and some estimates }\label{S3}
In this section, we establish the local well-posedness of the system \eqref{1} and collect several useful estimates for solutions that will be applied in subsequent sections. We begin with the following lemma, which establishes the local existence of classical solutions. Since the proof follows standard arguments based on the contraction mapping principle (see \cite{Tao+Winkler2011} for nonlinear diffusion and \cite{FWY} for signal-dependent sensitivity), we omit the details here to avoid redundancy.

\begin{lemma} \label{local}
    Let \(\Omega \subset \mathbb{R}^n\), \(n \geq 1\), be a bounded domain with smooth boundary.  Assume that \(a>0\), \(m>1\), \(\chi>0\), \(k>0\), and that \eqref{initial} holds. Then there exist a maximal existence time \(T_{\mathrm{max}} \in (0,\infty]\) and a unique pair of positive functions \((u,v)\) such that
    \[
        u, v \in C^0 \left(\overline{\Omega} \times [0,T_{\mathrm{max}})\right)
        \cap C^{2,1} \left(\overline{\Omega} \times (0,T_{\mathrm{max}})\right),
    \]
    and such that \((u,v)\) solves the system \eqref{1} with the boundary conditions \eqref{bdry} in the classical sense. Moreover, if \(T_{\mathrm{max}} < \infty\), then
    \begin{align}
        \limsup_{t \to T_{\mathrm{max}}}
        \left\{
        \| u(\cdot,t) \|_{L^\infty(\Omega)}
        +
        \| v(\cdot,t) \|_{W^{1,\infty}(\Omega)}
        \right\}
        = \infty. \label{local-exist-para-1}
    \end{align}
\end{lemma}
From now on, \((u,v)\) will denote the maximal classical solution of \eqref{1} with boundary condition \eqref{bdry} and initial data \eqref{initial}, as specified in Lemma \ref{local}. We also assume throughout that all hypotheses in Lemma \ref{local} are satisfied. The following lemma gives us the uniform in time \(L^1\) boundedness of solutions for \(u\) and \(v\).

\begin{lemma} \label{L1est}
    For any \(t\in (0,T_{\rm max})\), we have
    \begin{align} \label{L1est-1}
        \int_\Omega u(\cdot,t) = \int_\Omega u_0,
    \end{align}
    and 
    \begin{align} \label{L1est-2}
        \int_\Omega v(\cdot,t) \leq \max \left \{\int_\Omega u_0, \int_\Omega v_0 \right \}.
    \end{align}
\end{lemma}
\begin{proof}
  Integrating the first equation in \eqref{1} over \(\Omega\) immediately gives \eqref{L1est-1}. Integrating the second equation in \eqref{1} over \(\Omega\), using \eqref{L1est-1}, and applying Gronwall's inequality then yields \eqref{L1est-2}.
\end{proof}

As a consequence of the previous lemma and Lemma \ref{C52.Para-Reg}, we can improve the regularity of \(v\) in the following lemma.

\begin{lemma} \label{v'}
  Let \(\Omega \subset \mathbb{R}^n\), \(n \geq 2\), be a smoothly bounded domain.  For any \(p \in \left[1, \frac{n}{(n-2)_+} \right )\), there exists \(C=C(p)>0\) such that 
  \begin{align*}
      \int_\Omega v^p(\cdot,t) \leq C \qquad \text{for all }t\in (0,T_{\rm max}).
  \end{align*}
\end{lemma}
\begin{proof}
    From \eqref{L1est-1}, we have \(u \in L^\infty \left ((0,T_{\rm max}); L^1(\Omega) \right )\). This, together with Lemma \ref{C52.Para-Reg}, implies that \(v \in L^\infty \left ((0,T_{\rm max}); W^{1,q}(\Omega) \right )\) for any \(1\leq q< \frac{n}{n-1}\). For any \(p \in \left[1, \frac{n}{(n-2)_+} \right )\), we set \(q= \frac{np}{n+p}\), and then find that \(q \in \left [1, \frac{n}{n-1} \right )\). In light of Sobolev's inequality and \eqref{L1est-2}, we can find \(c_1>0\) and \(c_2>0\) such that
    \begin{align*}
        \left \| v(\cdot,t) \right \|_{L^p(\Omega)} \leq c_1  \left \| \nabla v(\cdot,t) \right \|_{L^{\frac{np}{n+p}}(\Omega)} + c_2 \left \| v(\cdot,t) \right \|_{L^1(\Omega)} \leq c_2
    \end{align*}
    for all \(t\in (0,T_{\rm max})\), which completes the proof.
\end{proof}

The next lemma gives us a time-independent lower bound for \(v\), which was established in \cite[Lemma 3.1]{Fujie_2018}. Here, we want to emphasize that the convexity assumption on \(\Omega\) is not needed.
\begin{lemma} \label{low}
     Let \(\Omega \subset \mathbb{R}^n\) with \(n \geq 2\) be a smoothly bounded domain. There exists a positive constant \(\eta\) depending only on \(\inf_{\bar{\Omega}} v_0, \int_\Omega u_0\), and \(\Omega\) such that 
    \begin{align*}
        \inf_{x\in \Omega} v(x,t) \geq \eta \qquad \text{for all }t\in (0,T_{\rm max}). 
    \end{align*}
\end{lemma}

A new differential inequality is provided in the following lemma, which enables us to prove the \(L^p\) boundedness of solutions. By virtue of this inequality, we can also establish an \(L^\infty\) boundedness result for the case \(k=0\), thus providing a proof different from those given in \cite{TaoWinkler2012}.

\begin{lemma} \label{L}
    Let \(\Omega \subset \mathbb{R}^n\) with \(n \geq 2\) be a bounded convex domain with smooth boundary. For any \(p > \max \left \{ 20,  \sqrt{n}-4 \right \}\), the following holds
    \begin{align*}
        \frac{d}{dt} \int_\Omega \frac{|\nabla v|^{2p}}{v^{p-1}}+ (p+1)\int_\Omega \frac{|\nabla v|^{2p}}{v^{p-1}} + \frac{p}{18} \int_\Omega \frac{|\nabla v|^{2p+2}}{v^{p+1}} \leq C \int_\Omega u^{p+1} \qquad \text{for all }t\in (0,T_{\rm max}), 
    \end{align*}
    where 
    \begin{align*}
        C= \frac{\left ( (2p-1)(p-1)18^{\frac{p}{p+1}} +6p(2p+\sqrt{n}-2)18^{\frac{p-1}{2(p+1)}} \right )^{p+1}}{p+1}.
    \end{align*}
\end{lemma}
\begin{proof}
    By direct calculations, we have
    \begin{align} \label{L.1}
        \frac{d}{dt}\int_\Omega \frac{|\nabla v|^{2p}}{v^{p-1}}&= 2p \int_\Omega \frac{|\nabla v|^{2p-2} \nabla v \cdot \nabla v_t}{v^{p-1}} -(p-1)\int_\Omega \frac{|\nabla v|^{2p}}{v^p} v_t \notag \\
        &= I+J.
    \end{align}
    Using the identity that 
    \begin{align*}
        \nabla v \cdot \nabla \Delta v = \frac{1}{2} \Delta |\nabla v|^2 -|D^2 v|^2,
    \end{align*}
 integrating by parts and noting that \(\frac{\partial |\nabla v|^2}{\partial \nu } \leq 0\) on \(\partial \Omega\) due to the convexity of \(\Omega\) (see \cite{Lion}), we obtain
    \begin{align} \label{L.2}
        I&= 2p \int_\Omega \frac{|\nabla v|^{2p-2} \nabla v \cdot \nabla (\Delta v-v+u)}{v^{p-1}} \notag \\
        &= 2p \int_\Omega \frac{|\nabla v|^{2p-2}}{v^{p-1}} \left ( \frac{1}{2} \Delta |\nabla v|^2- |D^2 v|^2 \right )-2p \int_\Omega \frac{|\nabla v|^{2p}}{v^{p-1}} + 2p \int_\Omega \frac{|\nabla v|^{2p-2} \nabla v \cdot \nabla u}{v^{p-1}} \notag \\
        &= -p \int_\Omega \nabla \left ( \frac{|\nabla v|^{2p-2}}{v^{p-1}} \right ) \cdot \nabla |\nabla v|^2 + p \int_{\partial \Omega} \frac{|\nabla v|^{2p-2}}{v^{p-1}} \frac{\partial |\nabla v|^2}{\partial \nu }\, dS   - 2p \int_\Omega  \frac{|\nabla v|^{2p-2}|D^2 v|^2}{v^{p-1}}  \notag \\
        &\quad-2p \int_\Omega \frac{|\nabla v|^{2p}}{v^{p-1}}- 2p \int_\Omega \nabla \left ( \frac{|\nabla v|^{2p-2} \nabla v}{v^{p-1}} \right )u \notag \\
        &\leq -p(p-1)\int_\Omega \frac{|\nabla v|^{2p-4}|\nabla |\nabla v|^2|^2}{v^{p-1}} +p(p-1) \int_\Omega \frac{|\nabla v|^{2p-2}\nabla v \cdot \nabla |\nabla v|^2}{v^p}  \notag \\
        &\quad-2p \int_\Omega \frac{|\nabla v|^{2p-2}|D^2 v|^2}{v^{p-1}}-2p \int_\Omega \frac{|\nabla v|^{2p}}{v^{p-1}}-2p(p-1) \int_\Omega \frac{|\nabla v|^{2p-4} \nabla |\nabla v|^2 \cdot \nabla v}{v^{p-1}}u \notag \\
        &\quad-2p \int_\Omega \frac{|\nabla v|^{2p-2} \Delta v}{v^{p-1}}u +2p(p-1) \int_\Omega \frac{|\nabla v|^{2p}}{v^p}u
    \end{align}
    for all \(t\in (0,T_{\rm max})\). By integration by parts, we infer that
    \begin{align} \label{L.3}
        J&= -(p-1) \int_\Omega \frac{|\nabla v|^{2p}(\Delta v-v+u)}{v^p} \notag\\
        &= (p-1) \int_\Omega \nabla \left ( \frac{|\nabla v|^{2p}}{v^p} \right ) \cdot \nabla v + (p-1)\int_\Omega \frac{|\nabla v|^{2p}}{v^{p-1}} -(p-1)\int_\Omega \frac{|\nabla v|^{2p}}{v^p} u \notag \\
        &= p(p-1)\int_\Omega \frac{|\nabla v|^{2p-2}\nabla v \cdot \nabla |\nabla v|^2}{v^p} -p(p-1)\int_\Omega \frac{|\nabla v|^{2p+2}}{v^{p+1}} \notag \\
        &\quad+ (p-1)\int_\Omega \frac{|\nabla v|^{2p}}{v^{p-1}} -(p-1)\int_\Omega \frac{|\nabla v|^{2p}}{v^p} u
    \end{align}
    for all \(t\in (0,T_{\rm max})\). Collecting \eqref{L.1}, \eqref{L.2} and \eqref{L.3} and invoking the following identity
    \begin{align*}
        v^2 |D^2 \ln v|^2 = |D^2 v|^2 - \frac{\nabla v \cdot \nabla|\nabla v|^2}{v} +\frac{|\nabla v|^4}{v^2},
    \end{align*}
    we write 
    \begin{align}\label{L.4}
        \frac{d}{dt}\int_\Omega \frac{|\nabla v|^{2p}}{v^{p-1}} &\leq  - p(p-1) \int_\Omega \frac{|\nabla v|^{2p-4}|\nabla |\nabla v|^2|}{v^{p-1}}- 2p \int_\Omega \frac{|\nabla v|^{2p-2}|D^2 v|^2}{v^{p-1}} \notag \\
        &\quad+ 2p(p-1)\int_\Omega \frac{|\nabla v|^{2p-2}\nabla v \cdot \nabla |\nabla v|^2}{v^p} - p(p-1)\int_\Omega \frac{|\nabla v|^{2p+2}}{v^{p+1}}- (p+1)\int_\Omega \frac{|\nabla v|^{2p}}{v^{p-1}} \notag \\
        &\quad+(2p-1)(p-1)\int_\Omega \frac{|\nabla v|^{2p}}{v^p}u -2p \int_\Omega \frac{|\nabla v|^{2p-2} \Delta v}{v^{p-1}}u+4p(p-1)\int_\Omega \frac{|\nabla v|^{2p-2}|D^2 v|}{v^{p-1}}u \notag \\
        &\leq - p(p-1) \int_\Omega \frac{|\nabla v|^{2p-4}|\nabla |\nabla v|^2|^2}{v^{p-1}} -2p \int_\Omega \frac{|\nabla v|^{2p-2}|D^2 \log v|^2}{v^{p-3}} -(p+1)\int_\Omega \frac{|\nabla v|^{2p}}{v^{p-1}}\notag \\
        &\quad+2p(p-2)\int_\Omega \frac{|\nabla v|^{2p-2}\nabla v \cdot \nabla |\nabla v|^2}{v^p} -p(p-3) \int_\Omega \frac{|\nabla v|^{2p+2}}{v^{p+1}} \notag \\
        &\quad+(2p-1)(p-1)\int_\Omega \frac{|\nabla v|^{2p}}{v^p}u +2p(2p-2+\sqrt{n}) \int_\Omega \frac{|\nabla v|^{2p-2}|D^2v|}{v^{p-1}}u
    \end{align}
    for all \(t\in (0,T_{\rm max})\). An application of Young's inequality yields
    \begin{align*}
        2p(p-2)\int_\Omega \frac{|\nabla v|^{2p-2}\nabla v \cdot \nabla |\nabla v|^2}{v^p} \leq p(p-1)\int_\Omega \frac{|\nabla v|^{2p-4}|\nabla |\nabla v|^2|^2}{v^{p-1}} + \frac{p(p-2)^2}{p-1} \int_\Omega  \frac{|\nabla v|^{2p+2}}{v^{p+1}}.
    \end{align*}
    This, together with \eqref{L.4}, implies that 
    \begin{align}\label{L.5}
         \frac{d}{dt}\int_\Omega \frac{|\nabla v|^{2p}}{v^{p-1}} &\leq-2p \int_\Omega \frac{|\nabla v|^{2p-2}|D^2 \log v|^2}{v^{p-3}} -(p+1)\int_\Omega \frac{|\nabla v|^{2p}}{v^{p-1}} +\frac{p}{p-1}\int_\Omega \frac{|\nabla v|^{2p+2}}{v^{p+1}} \notag \\
         &\quad +(2p-1)(p-1)\int_\Omega \frac{|\nabla v|^{2p}}{v^p}u +2p(2p-2+\sqrt{n}) \int_\Omega \frac{|\nabla v|^{2p-2}|D^2v|}{v^{p-1}}u
    \end{align}
    for all \(t \in (0,T_{\rm max})\). Abbreviate \[M=\int_\Omega \frac{|\nabla v|^{2p-2}|D^2 \ln v|^2}{v^{p-3}}. \] Using \eqref{Lm.2} and \eqref{Lm.3}, together with the fact that \(p>\sqrt{n}-4\), we obtain 
    \begin{align*}
        \frac{p+\sqrt{n}}{p+2} \leq 2.
    \end{align*}
    Hence, we infer that
    \begin{align}\label{L.6}
        \int_\Omega \frac{|\nabla v|^{2p+2}}{v^{p+1}} &\leq \left ( 4\left ( \frac{p+\sqrt{n}}{p+2} \right )^2+2 \right ) M \notag \\
        &\leq 18 M
    \end{align}
   and
    \begin{align}\label{L.7}
        \int_\Omega \frac{|\nabla v|^{2p-2}}{v^{p-1}}|D^2 v|^2 &\leq \left ( \frac{p+2}{p} \left ( \frac{p+\sqrt{n}}{p+2} \right )^2+1 \right )M \notag \\
        &\leq 9 M.
    \end{align}
    Applying Hölder's inequality, together with \eqref{L.6} and \eqref{L.7}, we deduce that
    \begin{align}\label{L.8}
         \int_\Omega \frac{|\nabla v|^{2p}}{v^p} u &\leq  \left (\int_\Omega \frac{|\nabla v|^{2p+2}}{v^{p+1}}  \right )^{\frac{p}{p+1}} \left ( \int_\Omega u^{p+1} \right )^{\frac{1}{p+1}} \notag \\
        &\leq 18^{\frac{p}{p+1}}  M^\frac{p}{p+1} \left ( \int_\Omega u^{p+1} \right )^{\frac{1}{p+1}}
    \end{align}
    and 
    \begin{align}\label{L.9}
        \int_\Omega \frac{|\nabla v|^{2p-2}|D^2v|}{v^{p-1}}u &\leq \left ( \int_\Omega \frac{|\nabla v|^{2p-2}|D^2v|^2}{v^{p-1}} \right )^\frac{1}{2} \left ( \int_\Omega \frac{|\nabla v|^{2p+2}}{v^{p+1}} \right )^{\frac{p-1}{2(p+1)}}\left ( \int_\Omega u^{p+1} \right )^{\frac{1}{p+1}} \notag \\
        &\leq 3 \cdot18^{\frac{p-1}{2(p+1)}}M^{\frac{p}{p+1}}  \left ( \int_\Omega u^{p+1} \right )^{\frac{1}{p+1}}.
    \end{align}
    Combining \eqref{L.8} and \eqref{L.9} and applying Young's inequality, we obtain 
    \begin{align}\label{L.10}
        (2p-1)(p-1) \int_\Omega \frac{|\nabla v|^{2p}}{v^p} u +2p(2p-2&+\sqrt{n})\int_\Omega \frac{|\nabla v|^{2p-2}|D^2v|}{v^{p-1}}u \leq c_1M^{\frac{p}{p+1}}  \left ( \int_\Omega u^{p+1} \right )^{\frac{1}{p+1}} \notag\\
        &\leq \frac{p}{p+1}M + \frac{c_1^{p+1}}{p+1} \int_\Omega u^{p+1} \notag \\
        &\leq M +\frac{c_1^{p+1}}{p+1}\int_\Omega u^{p+1}
    \end{align}
    with
    \begin{align*}
        c_1= (2p-1)(p-1)18^{\frac{p}{p+1}} +6p(2p+\sqrt{n}-2)18^{\frac{p-1}{2(p+1)}}.
    \end{align*}
   Collecting \eqref{L.5}, \eqref{L.6} and \eqref{L.10} and noting that
    \[
    \frac{p-1}{18} \geq \frac{p}{p-1},
    \]
    when \(p \geq 20\), we arrive at 
    \begin{align*}
        \frac{d}{dt}\int_\Omega \frac{|\nabla v|^{2p}}{v^{p-1}}+ (p+1)\int_\Omega\frac{|\nabla v|^{2p}}{v^{p-1}}   &\leq (-2p+1) M +  \frac{p}{p-1} \int_\Omega \frac{|\nabla v|^{2p+2}}{v^{p+1}}+\frac{c_1^{p+1}}{p+1}\int_\Omega u^{p+1} \notag \\
        &\leq - \left ( \frac{2p-1}{18} -\frac{p}{p-1}\right )\int_\Omega \frac{|\nabla v|^{2p+2}}{v^{p+1}}+\frac{c_1^{p+1}}{p+1}\int_\Omega u^{p+1} \notag \\
        &\leq -\frac{p}{18}\int_\Omega \frac{|\nabla v|^{2p+2}}{v^{p+1}}+\frac{c_1^{p+1}}{p+1}\int_\Omega u^{p+1}\qquad \text{for all }t\in (0,T_{\rm max}),
    \end{align*}
    which finishes the proof.
\end{proof}

\section[]{\(L^p\) boundedness of solutions in dimension two}\label{S4}

In this section, we shall derive several a priori estimates for solutions, including \(L\ln L\) and \(L^p\) boundedness for \(u\) in dimension two. Let us begin with the \(L\ln L\) estimate for \(u\) and the \(W^{1,2}\) estimate for \(v\) in the following lemma.

  \begin{lemma} \label{LlnL}
       Let \(\Omega \subset \mathbb{R}^2\) be a bounded domain with smooth boundary and let \(m>1\) and \(k>0\). There exists \(C>0\) such that 
    \begin{align} \label{LlnL-1}
        \int_\Omega u(\cdot,t) \log u(\cdot,t) \leq C \qquad \text{for all }t\in (0,T_{\rm max}),
    \end{align}
    and 
    \begin{align} \label{LlnL-2}
        \int_\Omega |\nabla v(\cdot,t)|^2  \leq C \qquad \text{for all }t\in (0,T_{\rm max}).
    \end{align}
  \end{lemma}

  \begin{proof}
      Multiplying the first equation of \eqref{1} by \(\log u +1\), integrating by parts and employing Lemma \ref{low}, we obtain 
      \begin{align} \label{LlnL.1}
          \frac{d}{dt} \int_\Omega u \log u &= - \int_\Omega (u+a)^{m-1}\frac{|\nabla u|^2}{u} +\chi \int_\Omega \nabla u \cdot \frac{\nabla v}{v^k} \notag \\
          &= - \int_\Omega (u+a)^{m-1}\frac{|\nabla u|^2}{u}  -\chi \int_\Omega u \frac{\Delta v}{v^k} +\chi k \int_\Omega u \frac{|\nabla v|^2}{v^{1+k}} \notag \\
          &\leq - \int_\Omega u^{m-2}|\nabla u|^2 +\frac{\chi }{\eta^k } \int_\Omega u |\Delta v| + \frac{\chi k}{\eta ^k} \int_\Omega u \frac{|\nabla v|^2}{v} \qquad \text{for all }t\in (0,T_{\rm max}),
      \end{align}
      where \(\eta >0\)  is given in Lemma \ref{low}. In view of Lemma \ref{LK-1} and Lemma \ref{v'}, we can find positive constants \(c_1,c_2\) and \(c_3\) such that 
      \begin{align*} 
          \int_\Omega \frac{|\nabla v|^4}{v^2} &\leq c_1 \int_\Omega |\Delta v|^2+c_2\int_\Omega v^2 \notag \\
          &\leq c_1 \int_\Omega |\Delta v|^2+c_3. 
      \end{align*}
      This, together with Young's inequality, implies that 
      \begin{align} \label{LlnL.2}
          \frac{\chi k}{\eta ^k} \int_\Omega u \frac{|\nabla v|^2}{v} &\leq \frac{1}{4c_1} \int_\Omega \frac{|\nabla v|^4}{v^2} + c_4 \int_\Omega u^2 \notag \\
          &\leq \frac{1}{4} \int_\Omega |\Delta v|^2 +c_4 \int_\Omega u^2+ c_5,
      \end{align}
      where \(c_4= \frac{\chi^2 k^2 c_1}{ \eta^{2k} }\) and \(c_5= \frac{ c_3}{4c_1}\). One more application of Young's inequality yields
      \begin{align}\label{LlnL.3}
        \frac{\chi }{\eta^k } \int_\Omega u |\Delta v| \leq \frac{1}{4} \int_\Omega |\Delta v|^2+ c_6 \int_\Omega u^2,
      \end{align}
      where \(c_6= \frac{\chi^2}{\eta^{2k} }\). Direct calculations show that 
      \begin{align} \label{LlnL.4}
          \frac{1}{2}\frac{d}{dt}\int_\Omega |\nabla v|^2 + \int_\Omega |\nabla v|^2  &=- \int_\Omega |\Delta v|^2 - \int_\Omega u \Delta v \notag \\
          &\leq -\frac{1}{2} \int_\Omega |\Delta v|^2 + \frac{1}{2}\int_\Omega u^2  \qquad \text{for all }t\in (0,T_{\rm max}).
      \end{align}
      Making use of \eqref{L1est-1} and applying the Gagliardo-Nirenberg interpolation inequality established in Lemma \ref{GN} , we obtain 
      \begin{align*} 
          \int_\Omega u^{m+1} &= \left \| u^\frac{m}{2} \right \|_{L^{\frac{2(m+1)}{m}}(\Omega)}^{\frac{2(m+1)}{m}} \notag \\
          &\leq c_7 \left \| \nabla u^\frac{m}{2} \right \|_{L^2(\Omega)}^{\frac{2(m+1)\theta}{m}}  \left \| u^\frac{m}{2} \right \|_{L^{\frac{2}{m}}(\Omega)}^{\frac{2(m+1)(1-\theta)}{m}}+c_7 \left \| u^\frac{m}{2} \right \|_{L^{\frac{2}{m}}(\Omega)}^{\frac{2(m+1)}{m}} \notag \\
          &\leq c_8 \left \| \nabla u^\frac{m}{2} \right \|_{L^2(\Omega)}^{\frac{2(m+1)\theta}{m}}  +c_8,
      \end{align*}
      where \(c_7>0\), \(c_8>0\) and \(\theta=\frac{m}{m+1}\). Since 
      \begin{align*}
          \frac{(m+1)\theta}{m}=1,
      \end{align*}
      the above estimate entails that there exists \(c_9>0\) satisfying
      \begin{align}\label{LlnL.5}
          - \int_\Omega u^{m-2}|\nabla u|^2  \leq - c_9 \int_\Omega u^{m+1} +c_9.
      \end{align}
      An application of Young's inequality yields 
      \begin{align} \label{LlnL.6}
          \left (c_4+c_6+\frac{1}{2} \right )\int_\Omega u^2 + \int_\Omega u \log u  \leq c_9 \int_\Omega u^{m+1}+c_{10},
      \end{align}
      where \(c_{10}>0\). Collecting from \eqref{LlnL.1} to \eqref{LlnL.6}, we arrive at 
      \begin{align*}
          \frac{d}{dt} \left \{ \int_\Omega u \log u+ \frac{1}{2}\int_\Omega |\nabla v|^2 \right \}+  \left \{ \int_\Omega u \log u+ \frac{1}{2}\int_\Omega |\nabla v|^2 \right \} \leq c_{11} \qquad \text{for all }t\in (0,T_{\rm max}),
      \end{align*}
      where \(c_{11}=c_5+c_9+c_{10}\). Finally, applying Gronwall's inequality to this immediately implies \eqref{LlnL-1} and \eqref{LlnL-2}, which completes the proof. 
  \end{proof}

The following lemma provides an essential inequality that is necessary for deriving the \(L^p\) boundedness of \(u\).

\begin{lemma} \label{C5.l}
     Let \(\Omega \subset \mathbb{R}^2\) be a smoothly bounded domain. For any \(p >1\), there exist positive constants \(C_1\), \(C_2\) and \(C_3\) depending on \(p\) such that 
    \begin{align} \label{C5.l-1}
        \frac{d}{dt}\int_\Omega |\nabla v|^{2p}+ \int_\Omega |\nabla v|^{2p}+C_1 \int_\Omega |\nabla v|^{2p+2}  \leq  C_1 \int_\Omega u^{p+1} +C_3 \quad \text{for all }t\in (0,T_{\rm max}).
    \end{align}
\end{lemma}

\begin{proof}
Several applications of integration by parts, together with the identity
\[
\Delta |\nabla v|^2 = 2\nabla v \cdot \nabla \Delta v + 2|D^2 v|^2,
\]
yield
\begin{align} \label{C5.l.1}
    \frac{1}{p} \frac{d}{dt}\int_\Omega |\nabla v|^{2p} &= 2\int_\Omega |\nabla v|^{2p-2} \nabla v \cdot \nabla \left ( \Delta v -  v + u \right ) \notag \\
    &= \int_\Omega |\nabla v|^{2p-2} \left ( \Delta |\nabla v|^2- 2|D^2 v|^2 \right )- 2\int_\Omega |\nabla v|^{2p}+2 \int_\Omega |\nabla v|^{2p-2} \nabla v \cdot \nabla u \notag \\
    &= -\frac{4(p-1)}{p^2} \int_\Omega  |\nabla |\nabla v|^p|^2 - 2 \int_\Omega |\nabla v|^{2p-2}|D^2 v|^2- 2 \int_\Omega |\nabla v|^{2p}\notag \\
    &\quad+2 \int_\Omega |\nabla v|^{2p-2} \nabla v \cdot \nabla u + \int_{\partial \Omega} |\nabla v|^{2p-2}\frac{\partial |\nabla v|^2}{\partial \nu} \qquad \text{for all }t\in (0,T_{\rm max}).
\end{align}
Thanks to Lemma 4.2 in \cite{Souplet-13}, we can find \(c_1>0\) such that 
   {\begin{align*}
       \frac{\partial |\nabla v|^2}{\partial \nu } \leq c_1 |\nabla v|^2 \quad \text{on } \partial \Omega.
   \end{align*}}
 This, together with the trace Sobolev's embedding theorem, \(W^{1,1}(\Omega) \to L^1(\partial \Omega)\), {and} Young's inequality, entails that
   \begin{align}\label{C5.l.2}
       \int_{\partial \Omega}|\nabla v|^{2(p-1)} \frac{\partial |\nabla v|^2}{\partial \nu} &\leq c_1  \int_{\partial \Omega}|\nabla v|^{2p} \notag \\
       &\leq c_2 \int_\Omega |\nabla v|^{2p}+ c_2 \int_\Omega |\nabla v|^{2p-2} |\nabla |\nabla v|^2| \notag \\
       &\leq \frac{p-1}{p^2}\int_\Omega |\nabla |\nabla v|^p|^2+c_3  \int_\Omega |\nabla v|^{2p},
   \end{align}
   where \(c_2>0\) and \(c_3>0\). From Lemma \ref{LlnL}[\eqref{LlnL-2}], we have that \(\int_\Omega |\nabla v|^2\) is bounded in \((0,T_{\rm max})\). This, in conjunction with the two-dimensional Gagliardo–Nirenberg interpolation inequality, implies that
   \begin{align*}
       \int_\Omega |\nabla v|^{2p+2} &= \left \| |\nabla v|^p \right \|_{L^{2+\frac{2}{p}}(\Omega)}^{2+\frac{2}{p}} \notag \\
       &\leq c_4 \left \| \nabla |\nabla v|^p \right \|_{L^2(\Omega)}^2 \left \| |\nabla v|^p \right \|_{L^{\frac{2}{p}}}^{\frac{2}{p}} +c_4\left \| |\nabla v|^p \right \|_{L^{\frac{2}{p}}}^{\frac{2(p+1)}{p}}  \notag \\
       &\leq c_5 \left \| \nabla |\nabla v|^p \right \|_{L^2(\Omega)}^2  +c_5,
   \end{align*}
   where \(c_4>0\) and \(c_5>0\). This entails that there exist \(c_6>0\) and \(c_7>0\) such that 
   \begin{align}\label{C5.l.3'}
       -\frac{2(p-1)}{p^2} \int_\Omega  |\nabla |\nabla v|^p|^2 \leq -c_6 \int_\Omega |\nabla v|^{2p+2} +c_7.
   \end{align}
   An application of Young's inequality yields the existence of \(c_8>0\) satisfying 
   \begin{align}\label{C5.l.3}
       \left (c_3+\frac{1}{p} \right  )  \int_\Omega |\nabla v|^{2p} \leq \frac{c_6}{4} \int_\Omega |\nabla v|^{2p+2} +c_8.
   \end{align}
   Integrating by parts and employing the inequality \( |\Delta v|^2 \leq 2|D^2v|^2\), we derive that
   \begin{align}\label{C5.l.4}
       2 \int_\Omega |\nabla v|^{2p-2} \nabla v \cdot \nabla u&= -2 \int_\Omega u |\nabla v|^{2p-2}\Delta v  - \frac{4(p-1) }{p} \int_\Omega |\nabla v|^{p-2} \nabla v \cdot \nabla |\nabla v|^p u \notag \\
       &\leq  \int_\Omega |\nabla v|^{2p-2}|\Delta v|^2 + \frac{p-1}{p^2}\int_\Omega |\nabla |\nabla v|^p|^2+c_7 \int_\Omega u^2|\nabla v|^{2p-2} \notag \\
       &\leq 2\int_\Omega |\nabla v|^{2p-2} |D^2 v|^2 + \frac{p-1}{p^2}\int_\Omega |\nabla |\nabla v|^p|^2+ \frac{c_6}{4} \int_\Omega |\nabla v|^{2p+2}+c_8 \int_\Omega u^{p+1}, 
   \end{align}
   where \(c_7>0\) and \(c_8>0\). Collecting from \eqref{C5.l.1} to \eqref{C5.l.4}, we arrive at
   \begin{align*}
       \frac{d}{dt} \int_\Omega |\nabla v|^{2p}+ \int_\Omega |\nabla v|^{2p} + \frac{c_6}{2} \int_\Omega |\nabla v|^{2p+2}\leq  c_8p\int_\Omega u^{p+1}+c_9 \quad \text{for all } t\in (0,T_{\rm max}), 
   \end{align*}
   where \(c_9= pc_7+pc_8\), which completes the proof.
\end{proof}

With the aid of the above lemmas, we are now ready to derive the \(L^p\) boundedness of \(u\), as stated in the following lemma.

\begin{lemma} \label{Lp2}
       Let \(\Omega \subset \mathbb{R}^2\) be a bounded domain with smooth boundary and let \(m>1\) and \(k>0\). For any \(p>1\), there exists \(C=C(p)>0\) such that 
    \begin{align} \label{Lp2-1}
        \int_\Omega u^p(\cdot,t) \leq C \qquad \text{for all }t\in (0,T_{\rm max}).
    \end{align}
  \end{lemma}
  \begin{proof}
     Without loss of generality, we assume that \(p>\max \left\{ m-1,1 \right \}\). Multiplying the first equation in \eqref{1} by \(u^{p-1}\), integrating by parts, and applying Young's inequality yields
\begin{align} \label{Lp2.1}
    \frac{1}{p} \frac{d}{dt}\int_\Omega u^p 
    &= -(p-1)\int_\Omega u^{p-2}(u+a)^{m-1}|\nabla u|^2
       +\chi (p-1) \int_\Omega u^{p-1} \nabla u \cdot \frac{\nabla v}{v^k} \notag \\
    &\leq -\frac{p-1}{2} \int_\Omega u^{p+m-3}|\nabla u|^2
       + \frac{\chi^2(p-1)}{2} \int_\Omega u^{p-m+1} \frac{|\nabla v|^2}{v^{2k}} \notag \\
    &\leq -\frac{p-1}{2} \int_\Omega u^{p+m-3}|\nabla u|^2
       +c_1 \int_\Omega u^{p-m+1} |\nabla v|^2 
       \qquad \text{for all } t\in (0,T_{\max}),
\end{align}
where \(c_1= \frac{\chi^2(p-1)}{2 \eta^{2k}}\), with \(\eta>0\) as provided in Lemma \ref{low}. In view of Lemma \ref{C5.l}, we can find positive constants \(c_2,c_3\) and \(c_4\) such that 
\begin{align} \label{Lp2.2}
    \frac{d}{dt}\int_\Omega |\nabla v|^{2p-2m+2} + \int_\Omega |\nabla v|^{2p-2m+2}+c_2 \int_\Omega |\nabla v|^{2p-2m+4} \leq c_3 \int_\Omega u^{p-m+2} +c_4 
\end{align}
for all $t\in (0,T_{\rm max})$. In light of Young's inequality, we can find \(c_5>0\) satisfying 
\begin{align} \label{Lp2.3}
    c_1 \int_\Omega u^{p-m+1} |\nabla v|^2  \leq c_2  \int_\Omega |\nabla v|^{2p-2m+4}  +c_5 \int_\Omega u^{p-m+2}.
\end{align}
Making use of the uniform boundedness of \(\int_\Omega u\) as established in \eqref{L1est-1} and applying Lemma \ref{GN}, we infer that 
\begin{align*}
    \int_\Omega u^{p+m} &= \left \| u^{\frac{p+m-1}{2}} \right \|_{L^{\frac{2(p+m)}{p+m-1}}(\Omega)}^{\frac{2(p+m)}{p+m-1}} \\
    & \leq c_6 \left \| \nabla u^{\frac{p+m-1}{2}} \right \|_{L^2(\Omega)}^2 \left \| u^{\frac{p+m-1}{2}} \right \|_{L^{\frac{2}{p+m-1}}(\Omega)}^{\frac{2}{p+m-1}} + c_6 \left \| u^{\frac{p+m-1}{2}} \right \|_{L^{\frac{2}{p+m-1}}(\Omega)}^{\frac{2(p+m)}{p+m-1}} \notag \\
    &\leq c_7 \int_\Omega u^{p+m-3}|\nabla u|^2 +c_8,
\end{align*}
where \(c_6>0\), \(c_7>0\) and \(c_8>0\). This implies that there exists \(c_9>0\) such that 
\begin{align} \label{Lp2.4}
    -\frac{p-1}{2} \int_\Omega u^{p+m-3}|\nabla u|^2 \leq -c_9 \int_\Omega u^{p+m}+c_9.
\end{align}
An application of Young's inequality yields the existence of \(c_{10}>0\) fulfilling
\begin{align} \label{Lp2.5}
    \frac{1}{p}\int_\Omega u^p+ (c_3+c_5) \int_\Omega u^{p-m+2} \leq c_9 \int_\Omega u^{p+m}+c_{10}.
\end{align}
Collecting from \eqref{Lp2.1} to \eqref{Lp2.5}, we arrive at 
\begin{align*}
    \frac{d}{dt} \left \{ \frac{1}{p}\int_\Omega u^p + \int_\Omega |\nabla v|^{2p-2m+2} \right \}+\frac{1}{p}\int_\Omega u^p + \int_\Omega |\nabla v|^{2p-2m+2} \leq c_{11} \qquad \text{for all } t\in (0,T_{\max}),
\end{align*}
where \(c_{11}= c_4+c_9+c_{10}\). This, together with Gronwall's inequality, implies that 
\begin{align*}
     \frac{1}{p}\int_\Omega u^p(\cdot,t) + \int_\Omega |\nabla v|^{2p-2m+2}(\cdot,t) \leq \max \left \{ c_{11},  \frac{1}{p}\int_\Omega u_0^p + \int_\Omega |\nabla v_0|^{2p-2m+2}\right \}\qquad \text{for all } t\in (0,T_{\max}),
\end{align*}
which completes the proof.
  \end{proof}

\section[]{\(L^p\) boundedness of solutions in higher dimensions} \label{S5}
This section establishes the \(L^p\) boundedness of \(u\) for any \(p>1\) when \(n\ge 3\), and is divided into three subsections. Subsections~\ref{5.1}, \ref{5.2}, and \ref{5.3} treat, respectively, the cases
\[
k<\frac12,\quad m>1+\frac{(1-k)(n-2)}{n};\qquad
\frac12\le k<1,\quad m>\frac{3n-2}{2n};\qquad
k\ge 1,\quad m>\max\left\{1,\frac{3n-2}{2n+2}\right\}.
\]

\subsection[Case k in (0,1/2), m > 1 + (1-k)(n-2)/n]
           {The case \(k \in \left(0,\frac{1}{2}\right)\) and 
            \(m> 1+ \frac{(1-k)(n-2)}{n}\)} \label{5.1}

Before deriving the \(L^p\) estimate for \(u\), we need the following auxiliary lemma.
            
\begin{lemma} \label{V}
 Let \(\Omega \subset \mathbb{R}^n\) with \(n \geq 3\) be a smoothly bounded domain  and let \(\gamma \in [1, \frac{n}{n-2})\). For any \(p> \gamma\), there exist \(C_1>0\) and \(C_2>0\) such that 
    \begin{align}  \label{v-1}
        \frac{d}{dt}\int_\Omega v^{p-\gamma+1}+\int_\Omega v^{p-\gamma+1}+  C_1 \int_\Omega v^p \leq C_2 \int_\Omega u^{\frac{p}{\gamma}}+C_2 \qquad \text{for all }t\in (0,T_{\rm max}).
    \end{align}
\end{lemma}
\begin{proof}
    The condition \(1\leq \gamma < \frac{n}{n-2} \) implies that \(0\leq \frac{n(\gamma-1)}{2}< \frac{n}{n-2}\). In view of Lemma \ref{v'}, we can find \(c_1>0\) such that 
    \begin{align} \label{v.1}
        \int_\Omega v^\frac{n(\gamma-1)}{2}(\cdot,t) \leq c_1 \qquad \text{for all }t\in (0,T_{\rm max}).
    \end{align}
   By employing Lemma \ref{GN}, we obtain
    \begin{align} \label{v.2}
        \int_\Omega v^p &= \left \|  v^{\frac{p-\gamma+1}{2}} \right \|_{L^{\frac{2p}{p-\gamma+1}}(\Omega)}^{\frac{2p}{p-\gamma+1}} \notag \\
        &\leq c_2 \left \| \nabla v^{\frac{p-\gamma+1}{2}}  \right \|_{L^2(\Omega)}^{\frac{2p}{p-\gamma+1} \theta} \left \| v^{\frac{p-\gamma+1}{2}}  \right \|_{L^{\frac{n(\gamma-1)}{p-\gamma+1}}(\Omega)}^{\frac{2p}{p-\gamma+1}(1-\theta)} +  c_2 \left \| v^{\frac{p-\gamma+1}{2}}  \right \|_{L^{\frac{2}{p-\gamma+1}}(\Omega)}^{\frac{2p}{p-\gamma+1}},
    \end{align}
    where \(c_2>0\) and 
    \begin{align*}
        \theta = \frac{\frac{p-\gamma+1}{n(\gamma-1)} -\frac{p-\gamma+1}{2p}}{\frac{p-\gamma+1}{n(\gamma-1)} +\frac{1}{n}- \frac{1}{2}} \in (0,1).
    \end{align*}
    Substituting the fact that \(\frac{p \theta}{p-\gamma+1}=1\) into the estimate \eqref{v.2} and combining with \eqref{v.1} and \eqref{L1est-2}, we deduce that  
    \begin{align} \label{v.3}
        -(p-\gamma+1)(p-\gamma)\int_\Omega v^{p-\gamma-1}|\nabla v|^2 \leq -c_3 \int_\Omega v^p +c_4,
    \end{align}
    where \(c_3>0\) and \(c_4>0\). Multiplying the second equation of \eqref{1} by \((p-\gamma+1)v^{p-\gamma}\), integrating by parts, using \eqref{v.3} and applying Young's inequality yields
    \begin{align*}
        \frac{d}{dt} \int_\Omega v^{p-\gamma+1} + (p-\gamma+1) \int_\Omega v^{p-\gamma+1} &= -(p-\gamma+1)(p-\gamma)\int_\Omega v^{p-\gamma-1}|\nabla v|^2 + (p-\gamma+1)\int_\Omega v^{p-\gamma}u \notag \\
        &\leq  -c_3 \int_\Omega v^p+(p-\gamma+1) \int_\Omega v^{p-\gamma}u+c_4 \notag \\
        &\leq -\frac{c_3}{2} \int_\Omega v^p + c_5 \int_\Omega u^\frac{p}{\gamma}+c_4 \qquad \text{for all } t\in (0,T_{\rm max}),
    \end{align*}
    where \(c_5>0\), which implies \eqref{v-1} since \(p-\gamma+1>1\). The proof is now complete.
\end{proof}

Thanks to the previous lemma and Lemma \ref{L}, we can now derive the \(L^p\) boundedness of solutions in the following lemma.
    
  \begin{lemma} \label{Lp'}
 Let \(\Omega \subset \mathbb{R}^n\) with \(n \geq 3\) be a smoothly convex bounded domain  and let \(k \in \left [0,  \frac{1}{2} \right )\) and \(m>1+ \frac{(1-k)(n-2)}{n}\). Then for any \(p>1\), there exists \(C>0\) such that 
    \begin{align*}
        \int_\Omega u^p(\cdot,t) \leq C \qquad \text{for all }t\in (0,T_{\rm max}).
    \end{align*}
\end{lemma}  

\begin{proof}
Without loss of generality, we assume that \(p> m-1+\max \left\{ 20, \sqrt{ n} -4 \right \}\). Multiplying the first equation of \eqref{1} by \(u^{p-1}\), integrating by parts and applying Young's inequality, we obtain 
    \begin{align}\label{Lp'.1}
        \frac{1}{p}\frac{d}{dt} \int_\Omega u^p &= - (p-1)\int_\Omega u^{p-2}(u+a)^{m-1}|\nabla u|^2 +\chi(p-1)\int_\Omega u^{p-1}\nabla u \cdot \frac{\nabla v}{v^k} \notag \\
        &\leq -\frac{p-1}{2} \int_\Omega u^{p+m-3}|\nabla u|^2 + \frac{\chi^2(p-1)}{2} \int_\Omega u^{p-m+1}\frac{|\nabla v|^2}{v^{2k}} \notag \\
        &= -\frac{2(p-1)}{(p+m-1)^2}\int_\Omega |\nabla u^{\frac{p+m-1}{2}}|^2
+\frac{\chi^2(p-1)}{2}\int_\Omega u^{p-m+1}\frac{|\nabla v|^2}{v^{2k}}
     \end{align}
 for all \(t\in (0,T_{\rm max})\).
Since
\[
m-1>\frac{(1-k)(n-2)}{n},
\]
we may choose \(\lambda\in\left[1,\frac{n}{n-2}\right)\) such that
\begin{align*}
\left(2(m-1)-\frac{n-2}{n}\right)\frac{1}{1-2k}
>\frac{1}{\lambda}
>\frac{n-2}{n}.
\end{align*}
In view of Lemma \ref{V}, we can find \(c_1>0\), \(c_2>0\), and \(c_3>0\) such that
\begin{align}\label{Lp'.3}
\frac{d}{dt} \int_\Omega v^{\lambda(p-m+1)+2-2k}
+\int_\Omega v^{\lambda(p-m+1)+2-2k}
+c_1 \int_\Omega v^{\lambda(p-m+2)+1-2k}
&\leq c_2 \int_\Omega u^{p-m+2+\frac{1-2k}{\lambda}}+c_3,
\end{align}
for all \(t\in(0,T_{\max})\). Since \(p-m+1+\frac{1-2k}{\lambda}> \max \left \{20, \sqrt{n}-4 \right \}\), we apply  Lemma \ref{L} to deduce that there exists \(c_4>0\) such that
\begin{align}\label{Lp'.4}
\frac{d}{dt} \int_\Omega \frac{|\nabla v|^{2(p-m+1+\frac{1-2k}{\lambda})}}{v^{p-m+\frac{1-2k}{\lambda}}}
&+\int_\Omega \frac{|\nabla v|^{2(p-m+1+\frac{1-2k}{\lambda})}}{v^{p-m+\frac{1-2k}{\lambda}}}
\notag \\
&+\frac{1}{18}\left(p-m+1+\frac{1-2k}{\lambda}\right)
\int_\Omega \frac{|\nabla v|^{2(p-m+2+\frac{1-2k}{\lambda})}}{v^{p-m+2+\frac{1-2k}{\lambda}}}
\notag \\
&\leq c_4 \int_\Omega u^{p-m+2+\frac{1-2k}{\lambda}}+c_4
\end{align}
for all \(t\in(0,T_{\max})\). Set
\[
q=p-m+2+\frac{1-2k}{\lambda}.
\]
An application of Young's inequality implies that there exists \(c_5>0\) such that
\begin{align}\label{Lp'.5}
\frac{\chi^2(p-1)}{2}\int_\Omega u^{p-m+1}\frac{|\nabla v|^2}{v^{2k}}
&=\frac{\chi^2(p-1)}{2}\int_\Omega u^{p-m+1}v^{1-2k}\frac{|\nabla v|^2}{v}
\notag \\
&\leq \frac{1}{18}\left(p-m+1+\frac{1-2k}{\lambda}\right)
\int_\Omega \frac{|\nabla v|^{2q}}{v^{q}}
\notag \\
&\quad +c_5\int_\Omega \left(u^{p-m+1}v^{1-2k}\right)^{\frac{q}{q-1}}.
\end{align}
Setting
\[
r=\frac{p-m+1+\frac{1-2k}{\lambda}}{p-m+1},
\]
we find that
\[
\frac{(p-m+1)qr}{q-1}
=\frac{r}{r-1}\cdot\frac{q(1-2k)}{(q-1)\lambda}
=p-m+2+\frac{1-2k}{\lambda}.
\]
Invoking Young's inequality again, we can find \(c_6>0\) satisfying
\begin{align}\label{Lp'.6}
c_5\int_\Omega \left(u^{p-m+1}v^{1-2k}\right)^{\frac{q}{q-1}}
&\leq c_1\int_\Omega v^{\frac{r}{r-1}\cdot\frac{q(1-2k)}{q-1}}
+c_6\int_\Omega u^{\frac{(p-m+1)qr}{q-1}}
\notag \\
&=c_1\int_\Omega v^{\lambda(p-m+2)+1-2k}
+c_6\int_\Omega u^{p-m+2+\frac{1-2k}{\lambda}}.
\end{align}
By the Gagliardo--Nirenberg interpolation inequality, we can find \(c_7>0\) such that
\begin{align}\label{Lp'.7}
\int_\Omega u^{p+m-1+\frac{2}{n}}
&=\left\|u^{\frac{p+m-1}{2}}\right\|_{L^{\frac{2(p+m-1+\frac{2}{n})}{p+m-1}}(\Omega)}
^{\frac{2(p+m-1+\frac{2}{n})}{p+m-1}}
\notag \\
&\leq c_7\left(
\left\|\nabla u^{\frac{p+m-1}{2}}\right\|_{L^2(\Omega)}^\theta
\left\|u^{\frac{p+m-1}{2}}\right\|_{L^{\frac{2}{p+m-1}}(\Omega)}^{1-\theta}
\right.
\notag \\
&\qquad\qquad
\left.
+\left\|u^{\frac{p+m-1}{2}}\right\|_{L^{\frac{2}{p+m-1}}(\Omega)}
\right)^{\frac{2(p+m-1+\frac{2}{n})}{p+m-1}},
\end{align}
where
\[
\theta=\frac{\frac{p+m-1}{2}-\frac{p+m-1}{2(p+m-1+\frac{2}{n})}}
{\frac{p+m-1}{2}+\frac{1}{n}-\frac{1}{2}}.
\]
Because \(u\in L^\infty((0,T_{\max});L^1(\Omega))\) due to Lemma \ref{L1est}[\eqref{L1est-1}] and
\[
\frac{2(p+m-1+\frac{2}{n})\theta}{p+m-1}=2,
\]
from \eqref{Lp'.7}, we can find \(c_8>0\) such that
\[
\int_\Omega u^{p+m-1+\frac{2}{n}}
\leq c_8\int_\Omega |\nabla u^{\frac{p+m-1}{2}}|^2+c_8.
\]
This further entails that
\begin{align}\label{Lp'.8}
c_9\int_\Omega u^{p+m-1+\frac{2}{n}}
\leq \frac{2(p-1)}{(p+m-1)^2}
\int_\Omega |\nabla u^{\frac{p+m-1}{2}}|^2+c_{10},
\end{align}
for some \(c_9>0\) and \(c_{10}>0\). Noting that
\[
p-m+2+\frac{1-2k}{\lambda}<p+m-1+\frac{2}{n}
\]
due to \(m-1>\frac{(n-2)(1-k)}{n}\), we apply Young's inequality and use \eqref{Lp'.8} to obtain
\begin{align}\label{Lp'.9}
(c_2+c_4+c_6)\int_\Omega u^{p-m+2+\frac{1-2k}{\lambda}}
+\frac{1}{p}\int_\Omega u^p
&\leq c_9\int_\Omega u^{p+m-1+\frac{2}{n}}+c_{11}
\notag \\
&\leq \frac{2(p-1)}{(p+m-1)^2}
\int_\Omega |\nabla u^{\frac{p+m-1}{2}}|^2+c_{12},
\end{align}
where \(c_{11}>0\) and \(c_{12}=c_{10}+c_{11}\). We define
\[
y(t):=\frac{1}{p}\int_\Omega u^p
+\int_\Omega \frac{|\nabla v|^{2(p-m+1+\frac{1-2k}{\lambda})}}
{v^{p-m+1+\frac{1-2k}{\lambda}}}
+\int_\Omega v^{\lambda(p-m+1)+2-2k}
\]
for all \(t\in(0,T_{\max})\). Collecting the estimates from \eqref{Lp'.1}--\eqref{Lp'.5} and \eqref{Lp'.9}, we arrive at
\[
y'(t)+y(t)\leq c_{13}
\qquad\text{for all }t\in(0,T_{\max}),
\]
where \(c_{13}=c_3+c_4+c_{12}\). Finally, applying Gronwall's inequality to the above estimate yields
\[
y(t)\leq \max\{y(0),c_{13}\}
\qquad\text{for all }t\in(0,T_{\max}),
\]
which completes the proof.
\end{proof}

\subsection[Case k in [1/2,1), m > (3n-2)/(2n)]
           {The case \(k \in \left[\frac{1}{2},1\right)\) and 
            \(m> \frac{3n-2}{2n}\)} \label{5.2}
By using similar arguments as in the proof of Lemma \ref{Lp2}, we can also derive the \(L^p\) boundedness of \(u\) when \(k \in \left [\frac{1}{2}, 1 \right )\) and \(m> \frac{3n-2}{2n} \). 
\begin{lemma} \label{lp}
    Let \(\Omega \subset \mathbb{R}^n\) with \(n \geq 3\) be a smoothly convex bounded domain  and let \(k \geq \frac{1}{2}\) and \(m> \frac{3n-2}{2n}\). Then for any \(p>1\), there exists \(C>0\) such that
    \[
        \int_\Omega u^p(\cdot,t) \leq C \qquad \text{for all } t\in (0,T_{\max}).
    \]
\end{lemma}

\begin{proof}
   Without loss of generality, we assume that \(p> m-1+\max \left\{ 20, \sqrt{ n} -4 \right \}\).  Recalling the estimate \eqref{Lp'.1}, we obtain 
    \begin{align} \label{lp.1}
           \frac{1}{p} \frac{d}{dt} \int_\Omega u^p
        &\leq  -\frac{2(p-1)}{(p+m-1)^2}\int_\Omega |\nabla u^{\frac{p+m-1}{2}}|^2
        + \frac{\chi^2(p-1)}{2} \int_\Omega u^{p-m+1} \frac{|\nabla v|^2}{v^{2k}} \qquad \text{for all }t\in (0,T_{\rm max}).
    \end{align}
    Lemma \ref{low} asserts that there exists \(\eta>0\) such that
    \[
        v(\cdot,t) \geq \eta  \qquad \text{for all } t\in (0,T_{\max}).
    \]
    Using this, noting that \(k \geq \frac{1}{2}\), and applying Young's inequality, we can find \(c_1>0\) such that
    \begin{align} \label{lp.2}
        \frac{\chi^2(p-1)}{2} \int_\Omega u^{p-m+1} \frac{|\nabla v|^2}{v^{2k}}
        &\leq \frac{\chi^2(p-1)}{2 \eta^{2k-1}} \int_\Omega u^{p-m+1} \frac{|\nabla v|^2}{v} \notag \\
        &\leq \frac{p-m+1}{18} \int_\Omega \frac{|\nabla v|^{2p-2m+4}}{v^{p-m+2}}
        + c_1 \int_\Omega u^{p-m+2}.
    \end{align}
    Because \(p-m+1> \max \left \{20, \sqrt{n}-4 \right \}\), we apply Lemma \ref{L} to infer that there exists \(c_2>0\) such that
    \begin{align} \label{lp.3}
        \frac{d}{dt} \int_\Omega \frac{|\nabla v|^{2p-2m+2}}{v^{p-m}}
        &+ \int_\Omega \frac{|\nabla v|^{2p-2m+2}}{v^{p-m}}
        + \frac{p-m+1}{18} \int_\Omega \frac{|\nabla v|^{2p-2m+4}}{v^{p-m+2}} \leq c_2 \int_\Omega u^{p-m+2}
    \end{align}
    for all \(t\in (0,T_{\max})\). Similar to the estimate \eqref{Lp'.8}, we can find \(c_3>0\) and \(c_4>0\) satisfying
\begin{align}\label{lp.4}
-\frac{2(p-1)}{(p+m-1)^2}
\int_\Omega |\nabla u^{\frac{p+m-1}{2}}|^2 \leq -c_3\int_\Omega u^{p+m-1+\frac{2}{n}}+ c_4.
\end{align}
 In light of Young's inequality, we can find \(c_5>0\) such that 
\begin{align}\label{lp.5}
    c_2 \int_\Omega u^{p-m+2}+ \frac{1}{p}\int_\Omega u^p \leq c_3\int_\Omega u^{p+m-1+\frac{2}{n}} +c_5.
\end{align}
    Combining the estimates from \eqref{lp.1} to \eqref{lp.5}, we arrive at
    \begin{align}\label{lp.8}
        \frac{d}{dt} \left\{ \frac{1}{p}\int_\Omega u^p
        + \int_\Omega \frac{|\nabla v|^{2p-2m+2}}{v^{p-m}} \right\}
        + \left\{ \frac{1}{p}\int_\Omega u^p
        + \int_\Omega \frac{|\nabla v|^{2p-2m+2}}{v^{p-m}} \right\}
        \leq c_6
    \end{align}
    for all \(t\in (0,T_{\max})\), where \(c_6=c_4+c_5\). Finally, applying Gronwall's inequality to \eqref{lp.8} yields
    \[
        \frac{1}{p} \int_\Omega u^p(\cdot,t)
        + \int_\Omega \frac{|\nabla v(\cdot,t)|^{2p-2m+2}}{v^{p-m}(\cdot,t)}
        \leq \max \left\{ c_7, \frac{1}{p}\int_\Omega u_0^p
        + \int_\Omega \frac{|\nabla v_0|^{2p-2m+2}}{v_0^{p-m}} \right\}
    \]
    for all \(t\in (0,T_{\max})\), which completes the proof.
\end{proof}
            
\subsection[Case k in [1,infinity), m > max{1,(3n-2)/(2n+2)}]
           {The case \(k \in \left[1,\infty\right)\) and 
            \(m> \max\left\{1,\frac{3n-2}{2n+2}\right\}\)} \label{5.3}
 
Let us begin with the following spatio-temporal \(L^2\) estimate for \( \nabla \ln (u+a)\) in the following lemma.
\begin{lemma} \label{L1}
 Let \(\Omega \subset \mathbb{R}^n\) with \(n \geq 3\) be a smoothly bounded domain, and let \(k \geq 1\) and \(m>1\).  There exists \(C>0\) such that 
    \begin{align*}
        \int_t^{t+\tau} \int_\Omega \frac{|\nabla u(\cdot,s)|^2}{(u(\cdot,s)+a)^2 } \, ds \leq C \qquad \text{for all }t\in (0,T_{\rm max}-\tau),
    \end{align*}
    where \(\tau := \min \left \{ 1, \frac{T_{\rm max}}{2} \right \}\).
\end{lemma}

\begin{proof}
    Multiplying the first equation of \eqref{1} by \(-\frac{1}{u+a}\), integrating by parts, applying Young's inequality and employing Lemma \ref{low} yields
    \begin{align} \label{L1.1}
        -\frac{d}{dt}\int_\Omega \log (u+a)&= -\int_\Omega \frac{|\nabla u|^2}{(u+a)^{3-m}}+\chi \int_\Omega \frac{u}{(u+a)^2} \nabla u \cdot \frac{\nabla v}{v^k} \notag \\
        &\leq -a^{m-1} \int_\Omega \frac{|\nabla u|^2}{(u+a)^2}+ \chi \int_\Omega \frac{|\nabla u|}{u+a} \frac{|\nabla v|}{v^k} \notag \\
        &\leq - \frac{a^{m-1}}{2}\int_\Omega \frac{|\nabla u|^2}{(u+a)^2} + \frac{\chi^2}{2 a^{m-1}} \int_\Omega \frac{|\nabla v|^2}{v^{2k}} \notag \\
        &\leq  - \frac{a^{m-1}}{2}\int_\Omega \frac{|\nabla u|^2}{(u+a)^2} +c_1 \int_\Omega \frac{|\nabla v|^2}{v^2} \qquad \text{for all }t\in (0,T_{\rm max}),
    \end{align}
    where \(c_1 = \frac{\chi^2}{2a^{m-1}\eta^{2k-2}}\) with \(\eta>0\) given in Lemma \ref{low}. Testing the second equation of \eqref{1} by \(-\frac{1}{v}\) and integrating by parts, we obtain 
    \begin{align} \label{L1.2}
        -\frac{d}{dt} \int_\Omega \log v &= -\int_\Omega \frac{|\nabla v|^2}{v^2} +|\Omega | - \int_\Omega \frac{u}{v}\qquad \text{for all }t\in (0,T_{\rm max}).
    \end{align}
    Combining \eqref{L1.1} and \eqref{L1.2}, we arrive at
    \begin{align*}
        - \frac{d}{dt} \left \{ \int_\Omega \log (u+a)+ c_1 \int_\Omega \log v \right \}+\frac{a^{m-1}}{2}\int_\Omega \frac{|\nabla u|^2}{(u+a)^2} \leq c_1|\Omega| \qquad \text{for all }t\in (0,T_{\rm max}).
    \end{align*}
    Integrating this over time, applying Lemma \ref{L1est} and Lemma \ref{low}, we deduce that  
    \begin{align*}
      \frac{a^{m-1}}{2}  \int_t^{t+\tau} \int_\Omega \frac{|\nabla u(\cdot,s)|^2}{(u(\cdot,s)+a)^2 } \, ds &\leq c_1|\Omega|\tau + \int_\Omega \log (u(\cdot,t+\tau)+a) +c_1 \int_\Omega \log v(\cdot,t+\tau ) \\
      &\quad - \int_\Omega \log (u(\cdot,t)+a) -c_1 \int_\Omega \log v(\cdot,t ) \notag \\
      &\leq c_1|\Omega|\tau + \int_\Omega (u(\cdot,t+\tau)+a) +c_1 \int_\Omega  v(\cdot,t+\tau ) 
      - |\Omega | \log a -c_1 |\Omega| \log \eta\\
      &\leq c_2 \qquad \text{for all }t\in (0,T_{\rm max}-\tau ),
    \end{align*}
    where \(c_2= c_1|\Omega|\tau + \int_\Omega u_0+a|\Omega|+ c_1 \max \left \{ \int_\Omega u_0, \int_\Omega v_0 \right \}- |\Omega | \log a -c_1 |\Omega| \log \eta\) with \(\eta >0\) given in Lemma \ref{low}. The proof is now complete.
  \end{proof}
Before deriving the \(L^p\) estimate for \(u\), we need the following auxiliary lemma, which shows that the \(L^2\) norm of \(\nabla u\) controls the time evolution of the weighted quantity \(\int_\Omega \frac{|\nabla v|^4}{v^2}\).

\begin{lemma} \label{L2}
    Let \(\Omega \subset \mathbb{R}^n\) with \(n \geq 3\) be a smoothly convex bounded domain. There exist positive constants \(C_1\) and \(C_2\) such that the following holds
    \begin{align} \label{L2-1}
        \frac{d}{dt}\int_\Omega \frac{|\nabla v|^4}{v^2}+\int_\Omega \frac{|\nabla v|^4}{v^2}+C_1\int_\Omega \frac{|\nabla v|^6}{v^4}+C_1 \int_\Omega \frac{|\nabla v|^2|D^2 v|^2}{v^2}   \leq C_2 \int_\Omega |\nabla u|^2 
    \end{align}
    for all \(t\in (0,T_{\rm max})\).
\end{lemma}
\begin{proof}
    By direct calculations, we obtain 
    \begin{align} \label{L2.1}
        \frac{d}{dt}\int_\Omega \frac{|\nabla v|^4}{v^2} &= 4\int_\Omega \frac{|\nabla v|^2 \nabla v \cdot \nabla v_t}{v_2} -2 \int_\Omega \frac{|\nabla v|^4}{v^3}v_t \notag \\
        &= I+J.
    \end{align}
    Making use of the second equation of \eqref{1},  applying the identity \(\nabla v \cdot \nabla \Delta v = \frac{1}{2} \Delta |\nabla v|^2- |D^2 v|^2 \), and integrating by parts several times, we deduce that
    \begin{align} \label{L2.2}
        I&= 4\int_\Omega \frac{|\nabla v|^2 \nabla v \cdot \nabla \left ( \Delta v-v+u \right )}{v^2} \notag \\
        &= 2 \int_\Omega \frac{|\nabla v|^2}{v^2} \Delta |\nabla v|^2 - 4 \int_\Omega \frac{|\nabla v|^2|D^2v|^2}{v^2} -4 \int_\Omega \frac{|\nabla v|^4}{v^2} +4\int_\Omega \frac{|\nabla v|^2 \nabla v \cdot \nabla u}{v^2} \notag \\
        &=-2 \int_\Omega \frac{ \left | \nabla |\nabla v|^2 \right |^2}{v^2}+4 \int_\Omega \frac{|\nabla v|^2 \nabla v \cdot \nabla |\nabla v|^2}{v^3} +2 \int_{\partial \Omega} \frac{|\nabla v|^2}{v^2} \cdot \frac{\partial |\nabla v|^2}{\partial \nu }\, dS \notag \\
        &\quad- 4 \int_\Omega \frac{|\nabla v|^2|D^2v|^2}{v^2}-4 \int_\Omega \frac{|\nabla v|^4}{v^2} +4\int_\Omega \frac{|\nabla v|^2 \nabla v \cdot \nabla u}{v^2}  \notag \\
        &\leq -2 \int_\Omega \frac{ \left | \nabla |\nabla v|^2 \right |^2}{v^2}+4 \int_\Omega \frac{|\nabla v|^2 \nabla v \cdot \nabla |\nabla v|^2}{v^3} - 4 \int_\Omega \frac{|\nabla v|^2|D^2v|^2}{v^2} \notag \\
        &\quad-4 \int_\Omega \frac{|\nabla v|^4}{v^2} +4\int_\Omega \frac{|\nabla v|^2 \nabla v \cdot \nabla u}{v^2} \qquad \text{for all } t \in (0,T_{\rm max}),
    \end{align}
    where the last inequality holds due to the fact that \(\frac{\partial |\nabla v|^2}{ \partial \nu} \leq 0\) on \(\partial \Omega\) by convexity of \(\Omega\) (see \cite{Lion}). Next, using the second equation of \eqref{1} and integrating by parts, we infer that 
    \begin{align} \label{L2.3}
        J&= - 2\int_\Omega \frac{|\nabla v|^4}{v^3} \left ( \Delta v-v+u \right ) \notag \\
        &= 4 \int_\Omega \frac{|\nabla v|^2 \nabla v \cdot \nabla |\nabla v|^2}{v^3}-6\int_\Omega \frac{|\nabla v|^6}{v^4}+2\int_\Omega \frac{|\nabla v|^4}{v^2} -2 \int_\Omega u \frac{|\nabla v|^4}{v^3}
    \end{align}
    for all $t\in (0,T_{\rm max})$. Collecting \eqref{L2.1}, \eqref{L2.2} and \eqref{L2.3}, employing the identity that 
      \begin{align*}
        |D^2v|^2= v^2 |D^2\ln v|^2+ \frac{1}{v} \nabla v \cdot \nabla |\nabla v|^2- \frac{|\nabla v|^4}{v^2},
    \end{align*}
 using \eqref{Lm.2} and \eqref{Lm.3} with \(q=4\) and \(l=2\) and applying Young's inequality, we arrive at
    \begin{align*}
             \frac{d}{dt}\int_\Omega \frac{|\nabla v|^4}{v^2}+ 2\int_\Omega \frac{|\nabla v|^4}{v^2} &\leq-2 \int_\Omega \frac{1}{v^2} \left | \nabla |\nabla v|^2- \frac{|\nabla v|^2 \nabla v}{v} \right |^2 - 4 \int_\Omega |D^2 \ln v|^2|\nabla v|^2  + 4\int_\Omega \frac{|\nabla v|^3|\nabla u|}{v^2} \notag \\
             &\leq - 4 \int_\Omega |D^2 \ln v|^2|\nabla v|^2  + 4\int_\Omega \frac{|\nabla v|^3|\nabla u|}{v^2} \notag \\
             &\leq -c_1 \int_\Omega \frac{|D^2v|^2|\nabla v|^2}{v^2} -c_2 \int_\Omega \frac{|\nabla v|^6}{v^4} + 4\int_\Omega \frac{|\nabla v|^3|\nabla u|}{v^2} \notag \\
             &\leq -c_1 \int_\Omega \frac{|D^2v|^2|\nabla v|^2}{v^2} -\frac{c_2}{2} \int_\Omega \frac{|\nabla v|^6}{v^4} + c_3\int_\Omega |\nabla u|^2
    \end{align*}
    for all \(t\in (0,T_{\rm max})\), where \(c_1= \frac{18}{4(3+\sqrt{n})^2+18}\), \(c_2= \frac{6}{(3+\sqrt{n})^2+3}\) and \(c_3=\frac{8}{c_2}\). This clearly implies \eqref{L2-1} with \(C_1= \min \left\{c_1, \frac{c_2}{2} \right \} \) and \(C_2=c_3\) and concludes the proof.
\end{proof}

The second key ingredient needed for the proof of the \(L^p\) boundedness of solutions is the following lemma. The time evolution of the weighted quantity
\(
\int_\Omega u\frac{|\nabla v|^2}{v}
\)
produces a favorable term
\(
-\int_\Omega u^2\frac{|\nabla v|^2}{v^2},
\)
which allows us to control the other nonlinear terms arising from the singular sensitivity.

\begin{lemma} \label{LM}
 Let \(\Omega \subset \mathbb{R}^n\) with \(n \geq 3\) be a smoothly convex bounded domain and  let \(a>0\), \(m>1\) , \(\chi>0\), \(k\geq 1\).  For any \(\varepsilon \in \left (0, \frac{1}{6} \right )\), there exists \(C=C(\varepsilon)>0\) such that 
    \begin{align*}
        \frac{d}{dt}\int_\Omega u \frac{|\nabla v|^2}{v} +\int_\Omega u \frac{|\nabla v|^2}{v}+ \frac{1}{2}\int_\Omega u^2 \frac{|\nabla v|^2}{v^2} &\leq \varepsilon \int_\Omega (u+a)^{2m-2}|\nabla u|^2 +C\int_\Omega \frac{|\nabla v|^2|D^2 v|^2}{v^2} \notag \\
        &\quad+ C\int_\Omega \frac{|\nabla v|^6}{v^4}+C \int_\Omega \frac{|\nabla u|^2}{(u+a)^2}
    \end{align*}
     for all \(t\in (0,T_{\rm max})\).
\end{lemma}
\begin{proof}
Direct calculations show that 
\begin{align} \label{LM.1}
    \frac{d}{dt}\int_\Omega  u \frac{|\nabla v|^2}{v} &= \int_\Omega \frac{|\nabla v|^2}{v}u_t + 2 \int_\Omega \frac{\nabla v \cdot \nabla v_t}{v}u - \int_\Omega u \frac{|\nabla v|^2}{v^2}v_t \notag \\
    &= I+J+K.
\end{align}
By using the identity \(\nabla |\nabla v|^2= 2D^2 v \cdot \nabla v\), and integrating by parts several times, we obtain 
\begin{align} \label{LM.2}
    I&= \int_\Omega \frac{|\nabla v|^2}{v} \left ( \nabla \cdot \left ( (u+a)^{m-1}\nabla u \right ) - \chi \nabla \cdot \left (  u \frac{ \nabla v}{v^k} \right ) \right ) \notag \\
    &=-\int_\Omega (u+a)^{m-1} \nabla \left ( \frac{|\nabla v|^2}{v} \right )\cdot \nabla u + \chi \int_\Omega u  \nabla \left ( \frac{|\nabla v|^2}{v} \right )\cdot \frac{\nabla v}{v^k} \notag \\
    &=- 2\int_\Omega (u+a)^{m-1} \frac{\nabla u \cdot (D^2 v \cdot \nabla v)}{v} + \int_\Omega (u+a)^{m-1} \frac{|\nabla v|^2 \nabla v \cdot \nabla u}{v^2} \notag \\
    &\quad+2 \chi \int_\Omega u \frac{\nabla v \cdot (D^2 v \cdot \nabla v)}{v^{1+k}} -\chi \int_\Omega u  \frac{|\nabla v|^4}{v^{2+k}}. 
\end{align}
By several applications of Young's inequality, we infer that 
\begin{align} \label{LM.3}
    2\int_\Omega (u+a)^{m-1} \frac{\nabla u \cdot (D^2 v \cdot \nabla v)}{v} \leq \frac{\varepsilon}{4}\int_\Omega (u+a)^{2m-2}|\nabla u|^2+ \frac{4}{\varepsilon} \int_\Omega \frac{|\nabla v|^2|D^2 v|^2}{v^2},
\end{align}
and 
\begin{align} \label{LM.4}
     \int_\Omega (u+a)^{m-1} \frac{|\nabla v|^2 \nabla v \cdot \nabla u}{v^2} \leq  \frac{\varepsilon}{4}\int_\Omega (u+a)^{2m-2}|\nabla u|^2+\frac{1}{\varepsilon}\int_\Omega \frac{|\nabla v|^6}{v^4},
\end{align}
as well as 
\begin{align} \label{LM.5}
    2 \chi \int_\Omega u \frac{\nabla v \cdot (D^2 v \cdot \nabla v)}{v^{1+k}} &\leq \varepsilon \int_\Omega u^2 \frac{|\nabla v|^2}{v^2} + \frac{\chi^2}{\varepsilon} \int_\Omega \frac{|D^2v|^2|\nabla v|^2}{v^{2k}} \notag \\
    &\leq \varepsilon \int_\Omega u^2 \frac{|\nabla v|^2}{v^2} + \frac{\chi^2}{\varepsilon \eta^{2k-2}} \int_\Omega \frac{|D^2v|^2|\nabla v|^2}{v^{2}},
\end{align}
where the last inequality holds because Lemma \ref{low} asserts that \(\inf_{x\in \Omega } v(x,t) \geq \eta\) for all \(t\in (0,T_{\rm max})\) for some \(\eta>0\). Collecting from \eqref{LM.2} to \eqref{LM.5}, we deduce that 
\begin{align}\label{LM.6}
    I \leq \frac{\varepsilon}{2}\int_\Omega (u+a)^{2m-2}|\nabla u|^2 + \varepsilon \int_\Omega u^2 \frac{|\nabla v|^2}{v^2}+c_1 \int_\Omega  \frac{|D^2v|^2|\nabla v|^2}{v^{2}} +c_1\int_\Omega \frac{|\nabla v|^6}{v^4},
\end{align}
where \(c_1=  \frac{4}{\varepsilon}+\frac{\chi^2}{\varepsilon \eta^{2k-2}}\). We employ the identities \(\nabla v \cdot \nabla \Delta v = \frac{1}{2} \Delta |\nabla v|^2- |D^2 v|^2 \) and \(\nabla |\nabla v|^2= 2D^2 v \cdot \nabla v\) and integrate by parts several times to infer that 
\begin{align}\label{LM.7}
    J&= 2 \int_\Omega \frac{u \nabla v \cdot \nabla \left ( \Delta v-v+u \right ) }{v} \notag \\
    &= \int_\Omega \frac{u}{v} \Delta |\nabla v|^2 -2 \int_\Omega u\frac{|\nabla v|^2}{v} +2 \int_\Omega u\frac{ \nabla v \cdot \nabla u}{v} \notag \\
    &= - 2\int_\Omega  \frac{\nabla u \cdot (D^2 v \cdot \nabla v)}{v}+ \int_\Omega \frac{u}{v^2} \nabla v \cdot (D^2v \cdot \nabla v) + \int_{\partial \Omega} \frac{u}{v} \cdot \frac{\partial |\nabla v|^2}{\partial \nu }\, dS \notag \\
    &\quad-2 \int_\Omega u\frac{|\nabla v|^2}{v}+2 \int_\Omega u\frac{ \nabla v \cdot \nabla u}{v} \notag \\
    &\leq 2 \int_\Omega \frac{|\nabla u||D^2 v||\nabla v|}{v} + \int_\Omega \frac{u|\nabla v|^2 |D^2 v|}{v^2 } + 2\int_\Omega \frac{u|\nabla v||\nabla u|}{v}-2 \int_\Omega u\frac{|\nabla v|^2}{v},
\end{align}
where the last inequality holds due to the positivity of \(u\) and \(v\) and the fact that \(\frac{\partial |\nabla v|^2}{ \partial \nu} \leq 0\) on \(\partial \Omega\) by convexity of \(\Omega\) \cite{Lion}. Several applications of Young's inequality yield
\begin{align}\label{LM.8}
    2 \int_\Omega \frac{|\nabla u||D^2 v||\nabla v|}{v} \leq \int_\Omega |\nabla u|^2 + \int_\Omega \frac{|D^2 v|^2 |\nabla v|^2}{v^2},
\end{align}
and 
\begin{align}\label{LM.9}
    \int_\Omega \frac{u|\nabla v|^2 |D^2 v|}{v^2 }  \leq \frac{\varepsilon}{2} \int_\Omega u^2\frac{|\nabla v|^2}{v^2} + \frac{1}{2\varepsilon} \int_\Omega \frac{|\nabla v|^2|D^2v|^2}{v^2},
\end{align}
and
\begin{align}\label{LM.10}
    2\int_\Omega \frac{u|\nabla v||\nabla u|}{v} \leq \frac{\varepsilon}{2} \int_\Omega u^2\frac{|\nabla v|^2}{v^2} + \frac{2}{\varepsilon}\int_\Omega |\nabla u|^2,
\end{align}
as well as 
\begin{align}\label{LM.11}
    \left (1+\frac{2}{\varepsilon} \right )\int_\Omega |\nabla u|^2 \leq \frac{\varepsilon}{2}\int_\Omega (u+a)^{2m-2}|\nabla u|^2 + c_2 \int_\Omega \frac{|\nabla u|^2}{(u+a)^2},
\end{align}
where \(c_2=c_2(\varepsilon)>0\). Combining the estimates from \eqref{LM.7} to \eqref{LM.11}, we obtain 
\begin{align}\label{LM.12}
    J \leq \frac{\varepsilon}{2}\int_\Omega (u+a)^{2m-2}|\nabla u|^2 + \varepsilon \int_\Omega u^2\frac{|\nabla v|^2}{v^2} +c_3 \int_\Omega \frac{|D^2 v|^2 |\nabla v|^2}{v^2}+c_2 \int_\Omega \frac{|\nabla u|^2}{(u+a)^2},
\end{align}
where \(c_3= 1+\frac{1}{2\varepsilon}\). Using the second equation of \eqref{1} and applying the inequality that \(|\Delta v| \leq \sqrt{n} |D^2v|\) and Young's inequality, we have
\begin{align}\label{LM.13}
K&= -\int_\Omega u \frac{|\nabla v|^2}{v^2} \left ( \Delta v-v +u \right ) \notag \\
&\leq \sqrt{n} \int_\Omega u \frac{|\nabla v|^2 |D^2 v|}{v^2} + \int_\Omega u \frac{|\nabla v|^2}{v} - \int_\Omega u^2 \frac{|\nabla v|^2}{v^2}  \notag \\
&\leq \left (-1+\varepsilon \right)\int_\Omega u^2 \frac{|\nabla v|^2}{v^2} + \frac{n}{4\varepsilon} \int_\Omega \frac{|\nabla v|^2 |D^2 v|^2}{v^2}  + \int_\Omega u \frac{|\nabla v|^2}{v}. 
\end{align}
Collecting \eqref{LM.1}, \eqref{LM.6}, \eqref{LM.12} and \eqref{LM.13}, we arrive at 
\begin{align*}
    \frac{d}{dt}\int_\Omega u \frac{|\nabla v|^2}{v} +\int_\Omega u \frac{|\nabla v|^2}{v} &\leq \varepsilon \int_\Omega (u+a)^{2m-2}|\nabla u|^2 +(-1+3\varepsilon )\int_\Omega u^2 \frac{|\nabla v|^2}{v^2} \notag \\
    &\quad+c_1\int_\Omega \frac{|\nabla v|^6}{v^4}+c_4\int_\Omega \frac{|\nabla v|^2 |D^2 v|^2}{v^2}+c_2 \int_\Omega \frac{|\nabla u|^2}{(u+a)^2}  \notag \\
    &\leq \varepsilon \int_\Omega (u+a)^{2m-2}|\nabla u|^2 -\frac{1}{2}\int_\Omega u^2 \frac{|\nabla v|^2}{v^2} \notag \\
    &\quad+c_5\int_\Omega \frac{|\nabla v|^6}{v^4}+c_5\int_\Omega \frac{|\nabla v|^2 |D^2 v|^2}{v^2}+c_5 \int_\Omega \frac{|\nabla u|^2}{(u+a)^2} 
\end{align*}
for all \(t\in (0,T_{\rm max})\), where \(c_4= c_1+c_3+\frac{n}{4\varepsilon}\) and \(c_5= \max \left \{c_1,c_2,c_4 \right \}.\) The proof is now complete.
\end{proof}
Thanks to the two main ingredients established in Lemmas~\ref{L2} and~\ref{LM}, we can now establish uniform-in-time \(L^{m+1}\) boundedness for \(u\), as stated below.

\begin{lemma} \label{L3}
 Let \(\Omega \subset \mathbb{R}^n\) with \(n \geq 3\) be a smoothly convex bounded domain  let \(k \geq 1\) and \(m>1\). There exists \(C>0\) such that 
    \begin{align} \label{L3-1}
        \int_\Omega u^{m+1}(\cdot,t) \leq C \qquad \text{for all }t\in (0,T_{\rm max}).
    \end{align}
\end{lemma}
\begin{proof}
    Multiplying the first equation of \eqref{1} by \((u+a)^m\), integrating by parts, applying Young's inequality and employing Lemma \ref{low}, we obtain 
    \begin{align} \label{L3.1}
        \frac{1}{m+1}\frac{d}{dt}\int_\Omega (u+a)^{m+1} &=-m \int_\Omega (u+a)^{2m-2}|\nabla u|^2+\chi m \int_\Omega (u+a)^{m-1}u \nabla u \cdot \frac{\nabla v}{v^k} \notag \\
        &\leq - \frac{m}{2}\int_\Omega (u+a)^{2m-2}|\nabla u|^2 + \frac{\chi^2 m}{2} \int_\Omega u^2 \frac{|\nabla v|^2}{v^{2k}} \notag \\
        &\leq  -\frac{m}{2}\int_\Omega (u+a)^{2m-2}|\nabla u|^2 + c_1 \int_\Omega u^2 \frac{|\nabla v|^2}{v^{2}}
    \end{align}
    for all $t\in (0,T_{\rm max})$, where \(c_1= \frac{\chi^2 m}{2 \eta^{2k-2}}\) with \(\eta>0\) given in Lemma \ref{low}. In view of Lemma \ref{L2}, we can find positive constants \(c_2\) and \(c_3\) such that 
    \begin{align} \label{L3.2}
        \frac{d}{dt}\int_\Omega \frac{|\nabla v|^4}{v^2}+\int_\Omega \frac{|\nabla v|^4}{v^2}+c_2\int_\Omega \frac{|\nabla v|^6}{v^4}+c_2 \int_\Omega \frac{|\nabla v|^2|D^2 v|^2}{v^2}   \leq c_3 \int_\Omega |\nabla u|^2  \qquad \text{for all }t\in (0,T_{\rm max}).
    \end{align}
    Applying Lemma \ref{LM} with \(\varepsilon := \min \left \{ \frac{1}{12}, \frac{m}{4(c_1+1)} \right \}\), we can find \(c_4>0\) such that 
    \begin{align}\label{L3.3}
        \frac{d}{dt}\int_\Omega u \frac{|\nabla v|^2}{v} +\int_\Omega u \frac{|\nabla v|^2}{v}+ \frac{1}{2}\int_\Omega u^2 \frac{|\nabla v|^2}{v^2} &\leq \varepsilon \int_\Omega (u+a)^{2m-2}|\nabla u|^2 +c_4\int_\Omega \frac{|\nabla v|^2|D^2 v|^2}{v^2} \notag \\
        &\quad+ c_4\int_\Omega \frac{|\nabla v|^6}{v^4}+c_4 \int_\Omega \frac{|\nabla u|^2}{(u+a)^2}
    \end{align}
     for all \(t\in (0,T_{\rm max})\). Let
     \begin{align*}
         F(t):= \frac{1}{m+1} \int_\Omega (u+a)^{m+1}+ A \int_\Omega \frac{|\nabla v|^4}{v^2} +B \int_\Omega u \frac{|\nabla v|^2}{v} \qquad \text{for all }t\in (0,T_{\rm max}),
     \end{align*}
     with \(A:= \frac{2c_1c_4}{c_2}\) and \(B:=2c_1\). An application of Young's inequality yields
     \begin{align}\label{L3.4}
         Ac_3 \int_\Omega |\nabla u|^2 \leq \varepsilon \int_\Omega (u+a)^{2m-2}|\nabla u|^2+c_5 \int_\Omega \frac{|\nabla u|^2}{(u+a)^2},
     \end{align}
     where \(c_5>0\). In light of Lemma \ref{GN}, Young's inequality and \eqref{L1est-1}, we deduce that 
     \begin{align}\label{L3.5}
         \frac{1}{m+1}\int_\Omega (u+a)^{m+1} &= \frac{1}{m+1} \left \| (u+a)^m \right \|_{L^{\frac{m+1}{m}}(\Omega)}^\frac{m+1}{m} \notag \\
         &\leq c_6 \left \|\nabla (u+a)^m \right \|_{L^2(\Omega)}^{\frac{(m+1)\theta}{m}} \left \| (u+a)^m \right \|_{L^{\frac{1}{m}}(\Omega)}^{\frac{(m+1)(1-\theta)}{m}} +c_6\left \| (u+a)^m \right \|_{L^{\frac{1}{m}}(\Omega)}^{\frac{m+1}{m}} \notag \\
         &\leq c_7 \left ( \int_\Omega (u+a)^{2m-2} |\nabla u|^2 \right )^\frac{(m+1)\theta}{2m} +c_8 \notag \\
         &\leq \varepsilon \int_\Omega (u+a)^{2m-2} |\nabla u|^2 +c_9,
     \end{align}
     where \(c_7,c_8,c_9\) are positive constants, 
     \begin{align*}
         \theta := \frac{m-\frac{m}{m+1}}{m+\frac{1}{n}-\frac{1}{2}} \in (0,1),
     \end{align*}
    and the last inequality holds due to \(\frac{(m+1)\theta}{2m}<1\). Collecting the estimates from \eqref{L3.1} to \eqref{L3.5}, we arrive at 
    \begin{align*}
        F'(t) +F(t) &\leq - \left ( \frac{m}{2}-(B+2) \varepsilon \right ) \int_\Omega (u+a)^{2m-2} |\nabla u|^2+ \left (c_1-\frac{B}{2} \right )\int_\Omega u^2 \frac{|\nabla v|^2}{v^2} \notag \\
        &\quad+ \left ( Bc_4-Ac_2 \right )\int_\Omega \frac{|\nabla v|^2|D^2 v|^2}{v^2} + \left ( Bc_4-Ac_2 \right ) \int_\Omega \frac{|\nabla v|^6}{v^4} \notag \\
        &\quad+c_{10}\int_\Omega \frac{|\nabla u|^2}{(u+a)^2} +c_9 \notag \\
        &\leq c_{10}\int_\Omega \frac{|\nabla u|^2}{(u+a)^2} +c_9 \qquad \text{for all }t\in (0,T_{\rm max}),
    \end{align*}
    where \(c_{10}=c_4+c_5\) and the last inequality holds because \(\varepsilon \leq \frac{m}{4(c_1+1)} =\frac{m}{2(B+2)}\), \(c_1= \frac{B}{2}\) and \(Bc_4=Ac_2\). By virtue of Lemma \ref{L1}, we can find \(c_{11}>0\) such that 
     \begin{align*}
        \int_t^{t+\tau} \int_\Omega \frac{|\nabla u|^2}{(u+a)^2 } \leq c_{11} \qquad \text{for all }t\in (0,T_{\rm max}-\tau),
    \end{align*}
    where \(\tau := \min \left \{ 1, \frac{T_{\rm max}}{2} \right \}\). Using this and applying Lemma \ref{ODI} with \(h(t):= \int_\Omega \frac{|\nabla u(\cdot,t)|^2}{(u(\cdot,t)+a)^2 }\), we can find \(c_{12}>0\) such that 
    \begin{align*}
       \frac{1}{m+1} \int_\Omega (u(\cdot,t)+a)^{m+1}+ A \int_\Omega \frac{|\nabla v(\cdot,t)|^4}{v^2(\cdot,t)} +B \int_\Omega u(\cdot,t) \frac{|\nabla v(\cdot,t)|^2}{v(\cdot,t)} \leq c_{12} \qquad \text{for all }t\in (0,T_{\rm max}),
    \end{align*}
    which implies \eqref{L3-1} and concludes the proof.
\end{proof}

With the assistance of Lemmas~\ref{L3} and~\ref{L}, we can now establish the uniform-in-time \(L^p\) boundedness of \(u\), as stated in the following lemma.
            
    \begin{lemma} \label{Lp}
        Let \(\Omega \subset \mathbb{R}^n\) with \(n \geq 2\) be a bounded convex domain with smooth boundary, let \(m> \max \left \{1, \frac{3n-2}{2n+2} \right \}\) and \(k \geq 1\). For any \(p>1\), there exists \(C=C(p)>0\) such that 
        \begin{align*}
            \int_\Omega u^p(\cdot,t) \leq  C \qquad \text{for all }t\in (0,T_{\rm max}).
        \end{align*}
    \end{lemma}

  \begin{proof}
    Without loss of generality, we assume that \(p> m-1 +\max \left\{20, \sqrt{n}-4 \right \}\). Multiplying the first equation by \((u+a)^{p-1}\), integrating by parts and applying Young's inequality yields
        \begin{align} \label{Lp.1}
            \frac{1}{p} \frac{d}{dt}\int_\Omega (u+a)^p &= -(p-1)\int_\Omega (u+a)^{p+m-3}|\nabla u|^2 +\chi (p-1)\int_\Omega u(u+a)^{p-2} \nabla u \cdot \frac{\nabla v}{v^k} \notag \\
            & \leq  -\frac{p-1}{2}\int_\Omega (u+a)^{p+m-3}|\nabla u|^2 +\frac{\chi^2(p-1)}{2} \int_\Omega (u+a)^{p-m+1}\frac{|\nabla v|^2}{v^{2k}} 
        \end{align}
        for all \(t\in (0,T_{\rm max})\). Applying Lemma \ref{low} and Young's inequality, we can find \(c_1>0\) such that
        \begin{align} \label{Lp.2}
            \frac{\chi^2(p-1)}{2} \int_\Omega (u+a)^{p-m+1}\frac{|\nabla v|^2}{v^{2k}} &\leq \frac{\chi^2(p-1)}{2\eta^{2k-1}} \int_\Omega (u+a)^{p-m+1}\frac{|\nabla v|^2}{v} \notag \\
            &\leq \frac{p-m+1}{18} \int_\Omega \frac{|\nabla v|^{2p-2m+4}}{v^{p-m+2}} +c_1 \int_\Omega (u+a)^{p-m+2},
        \end{align}
        where \(\eta>0\) is given in Lemma \ref{low}. Since \(p-m+1> \max \left\{20, \sqrt{n}-4 \right \}\), we can employ Lemma \ref{L} to deduce that 
        \begin{align} \label{Lp.3}
            \frac{d}{dt}\int_\Omega \frac{|\nabla v|^{2p-2m+2}}{v^{p-m}}+ \int_\Omega \frac{|\nabla v|^{2p-2m+2}}{v^{p-m}}+ \frac{p-m+1}{18}\int_\Omega \frac{|\nabla v|^{2p-2m+4}}{v^{p-m+2}}  \leq c_2 \int_\Omega u^{p-m+2}
        \end{align}
        for all \(t\in (0,T_{\rm max})\), where \(c_2>0\). By virtue of Lemma \ref{L3}, there exists \(c_3>0\) such that 
        \begin{align*}
            \int_\Omega (u(\cdot,t)+a)^{m+1} \leq c_3 \qquad \text{for all }t\in (0,T_{\rm max}).
        \end{align*}
        This, together with Lemma \ref{GN}, implies that 
        \begin{align*}
            \int_\Omega (u+a)^{p+m-1+\frac{2(m+1)}{n}} &= \left \| (u+a)^{\frac{p+m-1}{2}} \right \|_{L^{2+ \frac{4(m+1)}{n(p+m-1)}}(\Omega)}^{2+ \frac{4(m+1)}{n(p+m-1)}} \notag \\
            &\leq c_4 \left \| \nabla (u+a)^{\frac{p+m-1}{2}} \right \|_{L^2(\Omega)}^{\left (2+ \frac{4(m+1)}{n(p+m-1)} \right )\theta}  \left \|(u+a)^{\frac{p+m-1}{2}} \right \|_{L^{\frac{2(m+1)}{p+m-1}}(\Omega)}^{\left (2+ \frac{4(m+1)}{n(p+m-1)} \right )(1-\theta)} \notag \\
            &\quad+ c_4\left \|(u+a)^{\frac{p+m-1}{2}} \right \|_{L^{\frac{2(m+1)}{p+m-1}}(\Omega)}^{\frac{p+m-1}{m+1}+\frac{2}{n}} \notag \\
            &\leq c_5  \left \| \nabla (u+a)^{\frac{p+m-1}{2}} \right \|_{L^2(\Omega)}^{\left (2+ \frac{4(m+1)}{n(p+m-1)} \right )\theta}+ c_6,
        \end{align*}
        where \(c_4,c_5,c_6\) are positive constants and 
        \begin{align*}
            \theta = \frac{\frac{p+m-1}{2(m+1)}- \frac{n(p+m-1)}{2n(p+m-1)+4(m+1)}}{\frac{p+m-1}{2(m+1)} +\frac{1}{n}-\frac{1}{2}}.
        \end{align*}
       One can verify that
        \begin{align*}
            {\left (2+ \frac{4(m+1)}{n(p+m-1)} \right )\theta} =2. 
        \end{align*}
       This, together with the above estimate, entails that there exist positive constants \(c_7\) and \(c_8\) satisfying 
        \begin{align}\label{Lp.4}
            -\frac{p-1}{2}\int_\Omega (u+a)^{p+m-3}|\nabla u|^2 \leq -c_7 \int_\Omega (u+a)^{p+m-1+\frac{2(m+1)}{n}} +c_8.
        \end{align}
        Noting that  \(p-m+2< p+m-1+\frac{2(m+1)}{n}\) when \(m> \frac{3n-2}{2n+2}\) and \(p<p+m-1+\frac{2(m+1)}{n}\), we apply Young's inequality to obtain 
        \begin{align}\label{Lp.5}
            \frac{1}{p}\int_\Omega (u+a)^p+c_1\int_\Omega (u+a)^{p-m+2}+c_2 \int_\Omega u^{p-m+2} \leq c_7 \int_\Omega (u+a)^{p+m-1+\frac{2(m+1)}{n}} +c_9,
        \end{align}
        for some \(c_9>0\). Collecting all the above estimates from \eqref{Lp.1} to \eqref{Lp.5}, we arrive at 
        \begin{align*}
              \frac{d}{dt} \left \{ \frac{1}{p} \int_\Omega (u+a)^p+\int_\Omega \frac{|\nabla v|^{2p-2m+2}}{v^{p-m}} \right \}+\left \{ \frac{1}{p} \int_\Omega (u+a)^p+\int_\Omega \frac{|\nabla v|^{2p-2m+2}}{v^{p-m}} \right \} \leq c_{10} \qquad \text{for all }t\in (0,T_{\rm max}),
        \end{align*}
        where \(c_{10}=c_8+c_9\). Finally, an application of Gronwall's inequality yields
        \begin{align*}
            \frac{1}{p} \int_\Omega (u(\cdot,t)+a)^p+\int_\Omega \frac{|\nabla v(\cdot,t)|^{2p-2m+2}}{v^{p-m}(\cdot,t)} \leq c_{11}\qquad \text{for all }t\in (0,T_{\rm max}),
        \end{align*}
        where \(c_{11}=\max \left \{c_{10},  \frac{1}{p}\int_\Omega (u_0+a)^p+\int_\Omega \frac{|\nabla v_0|^{2p-2m+2}}{v_0^{p-m}} \right \}\), which finishes the proof.
  \end{proof}

\section{Proof of the main result} \label{S6} 
We are now ready to prove our main theorem.
\begin{proof}[Proof of Theorem \ref{thm1}]
    In view of Lemmas \ref{Lp2}, \ref{Lp'}, \ref{lp} and \ref{Lp}, there exists \(c_1>0\) such that
    \begin{align} \label{prf.1}
        \int_\Omega u^{n+1}(\cdot,t) \leq c_1 \qquad \text{for all }t\in (0,T_{\rm max}).
    \end{align}
    Applying Lemma \ref{C52.Para-Reg} and using \eqref{prf.1}, we can find \(c_2> 0\) such that
    \begin{align} \label{prf.2}
        \left \| v(\cdot,t) \right \|_{W^{1, \infty}(\Omega)} \leq c_2\qquad \text{for all }t\in (0,T_{\rm max}).
    \end{align}
    Recalling the estimate \eqref{Lp2.1} and using \eqref{prf.2}, we obtain 
    \begin{align}
        \frac{d}{dt} \int_\Omega u^p \leq - \frac{p(p-1)}{2} \int_\Omega u^{p+m-3}|\nabla u|^2 + \frac{\chi^2 p(p-1)c_2^2}{2 \eta^{2k}} \int_\Omega u^{p-m+1} \qquad \text{for all }t\in (0,T_{\rm max}),
    \end{align}
    where \(\eta>0\) is given in Lemma \ref{low}. Now, we apply a standard Moser iteration argument as established in \cite{Alikakos1, Alikakos2} for reaction-diffusion equation or in \cite{TaoWinkler2012}[Lemma A.1] for chemotaxis systems to deduce that 
    \begin{align} \label{prf.3}
          \left \| u(\cdot,t) \right \|_{L^{ \infty}(\Omega)} \leq c_3\qquad \text{for all }t\in (0,T_{\rm max}),
    \end{align}
    for some \(c_3>0\). Combining \eqref{prf.2}, \eqref{prf.3} and the extensibility criterion for solutions \eqref{local-exist-para-1}, we infer that \(T_{\rm max} = \infty\) and 
    \begin{align*}
        \sup_{t>0} \left \{  \left \| u(\cdot,t) \right \|_{L^{ \infty}(\Omega)} +\left \| v(\cdot,t) \right \|_{W^{1, \infty}(\Omega)}  \right \}< \infty,
    \end{align*}
    which finishes the proof.
\end{proof}

\section*{Declarations}
\paragraph{Conflict of Interest} The author declares that they have no conflict of interest.
\paragraph{Declaration of Generative AI and AI-Assisted Technologies in the Manuscript Preparation Process}

During the preparation of this work, the author used \textbf{DeepSeek} for language polishing, grammar correction, readability improvement, and checking the logical structure and mathematical computations in the proofs. The author independently verified all mathematical arguments and takes full responsibility for the content of the published article.

\paragraph{Acknowledgments} This work was supported by the Hangzhou Postdoctoral Research Grant.

     \paragraph{Data Availability}
 Data sharing is not applicable to this article as no datasets were generated or analyzed during
the current study.


\begin{thebibliography}{99}

\bibitem{Alikakos1}
N.~D. Alikakos.
\newblock An application of the invariance principle to reaction-diffusion equations.
\newblock \emph{J. Differ. Equ.}, 33:201--225, 1979.

\bibitem{Alikakos2}
N.~D. Alikakos.
\newblock \(L^p\) bounds of solutions of reaction-diffusion equations.
\newblock \emph{Commun. Partial Differ. Equ.}, 4:827--868, 1979.

\bibitem{Biler-1999}
P.~Biler.
\newblock Global solutions to some parabolic-elliptic systems of chemotaxis.
\newblock \emph{Adv. Math. Sci. Appl.}, 9:347--359, 1999.

\bibitem{Dolbeault1}
A.~Blanchet, J.~Dolbeault, and B.~Perthame.
\newblock Two-dimensional Keller-Segel model: Optimal critical mass and qualitative properties of the solutions.
\newblock \emph{Electron. J. Differ. Equ.}, 2006. Paper No. 44.

\bibitem{CS2012}
T.~Cie\'slak and C.~Stinner.
\newblock Finite-time blowup and global-in-time unbounded solutions to a parabolic-parabolic quasilinear Keller-Segel system in higher dimensions.
\newblock \emph{J. Differ. Equ.}, 252:5832--5851, 2012.

\bibitem{Ding2018}
M.~Ding.
\newblock Global boundedness in a fully parabolic quasilinear chemotaxis system with singular sensitivity.
\newblock \emph{J. Math. Anal. Appl.}, 461:1260--1270, 2018.

\bibitem{Dolbeault}
J.~Dolbeault and B.~Perthame.
\newblock An optimal critical mass in the two-dimensional Keller-Segel model in \(\mathbb{R}^2\).
\newblock \emph{C. R. Acad. Sci. Paris, Ser. I}, 339:611--616, 2004.

\bibitem{Fujie}
K.~Fujie.
\newblock Boundedness in a fully parabolic chemotaxis system with singular sensitivity.
\newblock \emph{J. Math. Anal. Appl.}, 424:675--684, 2015.

\bibitem{Fujie_2018}
K.~Fujie and T.~Senba.
\newblock A sufficient condition of sensitivity functions for boundedness of solutions to a parabolic-parabolic chemotaxis system.
\newblock \emph{Nonlinearity}, 31:1639--1672, 2018.

\bibitem{FWY}
K.~Fujie, M.~Winkler, and T.~Yokota.
\newblock Blow-up prevention by logistic sources in a parabolic-elliptic Keller-Segel system with singular sensitivity.
\newblock \emph{Nonlinear Anal.}, 109:56--71, 2014.

\bibitem{KMT}
K.~Fujie, M.~Winkler, and T.~Yokota.
\newblock Boundedness of solutions to parabolic-elliptic Keller-Segel systems with signal-dependent sensitivity.
\newblock \emph{Math. Methods Appl. Sci.}, 38:1212--1224, 2015.

\bibitem{DM}
D.~Horstmann and M.~Winkler.
\newblock Boundedness vs. blow-up in a chemotaxis system.
\newblock \emph{J. Differ. Equ.}, 215:52--107, 2005.

\bibitem{Jia2019}
Z.~Jia and Z.~Yang.
\newblock Global boundedness to a parabolic-parabolic chemotaxis model with nonlinear diffusion and singular sensitivity.
\newblock \emph{J. Math. Anal. Appl.}, 475:139--153, 2019.

\bibitem{Keller2}
E.~F. Keller and L.~A. Segel.
\newblock A model for chemotaxis.
\newblock \emph{J. Theor. Biol.}, 30:225--234, 1971.

\bibitem{Keller}
E.~F. Keller and L.~A. Segel.
\newblock Initiation of slime mold aggregation viewed as an instability.
\newblock \emph{J. Theor. Biol.}, 26:399--415, 1970.

\bibitem{Kurt}
H.~I. Kurt.
\newblock Boundedness in a chemotaxis system with weak singular sensitivity and logistic kinetics in any dimensional setting.
\newblock \emph{J. Differ. Equ.}, 416:1429--1461, 2025.

\bibitem{Kurt+Shen}
H.~I. Kurt and W.~Shen.
\newblock Finite-time blow-up prevention by logistic source in parabolic-elliptic chemotaxis models with singular sensitivity in any dimensional setting.
\newblock \emph{SIAM J. Math. Anal.}, 53:973--1003, 2021.

\bibitem{Lankeit_2016}
J.~Lankeit.
\newblock A new approach toward boundedness in a two-dimensional parabolic chemotaxis system with singular sensitivity.
\newblock \emph{Math. Methods Appl. Sci.}, 39:394--404, 2016.

\bibitem{Le2026NA}
M.~Le.
\newblock An improvement toward global boundedness in a fully parabolic chemotaxis with singular sensitivity in any dimension.
\newblock \emph{Nonlinear Anal.}, 268:Paper No. 114082, 13, 2026.

\bibitem{minh-preprint-2}
M.~Le.
\newblock Global boundedness in the higher-dimensional fully parabolic chemotaxis with weak singular sensitivity and logistic source.
\newblock \emph{Discrete Contin. Dyn. Syst. Ser. B}, 30:4858--4869, 2025.

\bibitem{LiLankeit2016}
Y.~Li and J.~Lankeit.
\newblock Boundedness in a chemotaxis-haptotaxis model with nonlinear diffusion.
\newblock \emph{Nonlinearity}, 29:1564--1595, 2016.

\bibitem{Lion}
P.-L. Lions.
\newblock R\'esolution de probl\`emes elliptiques quasilin\'eaires.
\newblock \emph{Arch. Ration. Mech. Anal.}, 74:335--353, 1980.

\bibitem{Souplet-13}
N.~Mizoguchi and P.~Souplet.
\newblock Nondegeneracy of blow-up points for the parabolic Keller-Segel system.
\newblock \emph{Ann. Inst. Henri Poincar\'e C Anal. Non Lin\'eaire}, 31, 2013.

\bibitem{Nagai1}
T.~Nagai.
\newblock Blow-up of radially symmetric solutions to a chemotaxis system.
\newblock \emph{Adv. Math. Sci. Appl.}, 5:581--601, 1995.

\bibitem{Nagai2}
T.~Nagai.
\newblock Global existence and blow-up of solutions to a chemotaxis system.
\newblock \emph{Nonlinear Anal.}, 47:777--787, 2001.

\bibitem{Nagai+Senba}
T.~Nagai and T.~Senba.
\newblock Behavior of radially symmetric solutions of a system related to chemotaxis.
\newblock \emph{Nonlinear Anal.}, 30:3837--3842, 1997.

\bibitem{Nagai4}
T.~Nagai, T.~Senba, and T.~Suzuki.
\newblock Blowup of nonradial solutions to parabolic-elliptic systems modeling chemotaxis in two-dimensional domains.
\newblock \emph{J. Inequal. Appl.}, 6:37--55, 2001.

\bibitem{Nagai3}
T.~Nagai, T.~Senba, and T.~Suzuki.
\newblock Chemotaxis collapse in a parabolic system of mathematical biology.
\newblock \emph{Hiroshima Math. J.}, 20:463--497, 2000.

\bibitem{NSY}
T.~Nagai, T.~Senba, and K.~Yoshida.
\newblock Application of the Trudinger-Moser inequality to a parabolic system of chemotaxis.
\newblock \emph{Funkcial. Ekvac.}, 40:411--433, 1997.

\bibitem{Ren2024}
G.~Ren and X.~Zhou.
\newblock Global boundedness in a two-dimensional chemotaxis system with nonlinear diffusion and singular sensitivity.
\newblock \emph{Adv. Nonlinear Anal.}, 13:Paper No. 20230125, 20, 2024.

\bibitem{Tao2011}
Y.~Tao.
\newblock Boundedness in a chemotaxis model with oxygen consumption by bacteria.
\newblock \emph{J. Math. Anal. Appl.}, 381:521--529, 2011.

\bibitem{Tao+Winkler2011}
Y.~Tao and M.~Winkler.
\newblock A chemotaxis-haptotaxis model: the roles of nonlinear diffusion and logistic source.
\newblock \emph{SIAM J. Math. Anal.}, 43:685--704, 2011.

\bibitem{TaoWinkler2012}
Y.~Tao and M.~Winkler.
\newblock Boundedness in a quasilinear parabolic-parabolic Keller-Segel system with subcritical sensitivity.
\newblock \emph{J. Differ. Equ.}, 252:692--715, 2012.

\bibitem{WangDingLi2017}
W.~Wang, M.~Y. Ding, and Y.~Li.
\newblock Global boundedness in a quasilinear chemotaxis system with general density-signal governed sensitivity.
\newblock \emph{J. Differ. Equ.}, 263:2851--2873, 2017.

\bibitem{Winkler+2019}
M.~Winkler.
\newblock A three-dimensional Keller-Segel-Navier-Stokes system with logistic source: global weak solutions and asymptotic stabilization.
\newblock \emph{J. Funct. Anal.}, 276:1339--1401, 2019.

\bibitem{Winkler-2010}
M.~Winkler.
\newblock Aggregation vs. global diffusive behavior in the higher-dimensional Keller-Segel model.
\newblock \emph{J. Differ. Equ.}, 248:2889--2905, 2010.

\bibitem{Winkler2022}
M.~Winkler.
\newblock Approaching logarithmic singularities in quasilinear chemotaxis-consumption systems with signal-dependent sensitivities.
\newblock \emph{Discrete Contin. Dyn. Syst. Ser. B}, 27:6565--6587, 2022.

\bibitem{Winkler-2013}
M.~Winkler.
\newblock Finite-time blow-up in the higher-dimensional parabolic-parabolic Keller-Segel system.
\newblock \emph{J. Math. Pures Appl.}, 100:748--767, 2013.

\bibitem{Winkler+2011}
M.~Winkler.
\newblock Global solutions in a fully parabolic chemotaxis system with singular sensitivity.
\newblock \emph{Math. Methods Appl. Sci.}, 34:176--190, 2011.

\end{thebibliography}
\end{document}